\documentclass[a4paper, 11pt]{article}
\usepackage[utf8]{inputenc}
\usepackage[T1]{fontenc}
\usepackage{textcomp}
\usepackage{amsmath, amssymb, bbm}
\usepackage{amsthm}
\usepackage[hmargin= 2.5 cm]{geometry}
\usepackage[colorlinks=true]{hyperref}
\usepackage[french]{babel}
\usepackage{mathrsfs}
\usepackage{tikz, pgf}
\usetikzlibrary{decorations}
\usetikzlibrary{decorations.pathmorphing}
\usetikzlibrary{shapes.geometric}
\usetikzlibrary{calc}
\usetikzlibrary{arrows}
\renewcommand{\textbf}[1]{\begingroup\bfseries\mathversion{bold}#1\endgroup}

\title{Classification of traveling waves for a quadratic Szeg\H{o} equation}
\author{Joseph \bsc{Thirouin}}
\date{February 6\up{th}, 2018}

\DeclareMathOperator{\im}{Im}
\DeclareMathOperator{\re}{Re}
\DeclareMathOperator{\tr}{Tr}

\DeclareMathOperator{\rg}{rk}

\DeclareMathOperator{\ran}{Ran}

\newtheorem{thm}{Theorem}
\newtheorem{prop}{Proposition}[section]
\newtheorem{lemme}[prop]{Lemma}
\newtheorem{cor}[prop]{Corollary}
\theoremstyle{definition}
\newtheorem*{déf}{Definition}
\theoremstyle{remark}
\newtheorem{rem}{Remark}
\newtheorem*{eg}{Example}

\usepackage{enumitem}
\setenumerate[1]{label=(\roman*)}

\begin{document}

\renewcommand{\refname}{Bibliography}
\renewcommand{\abstractname}{Abstract}
\maketitle

\begin{abstract}
We give a complete classification of the traveling waves of the following quadratic Szeg\H{o} equation :
\[i \partial_t u = 2J\Pi(|u|^2)+\bar{J}u^2, \quad u(0, \cdot)=u_0,\]
and we show that they are given by two families of rational functions, one of which is generated by a stable ground state. We prove that the other branch is orbitally unstable. 
\end{abstract}

\section{Introduction}

This paper is devoted to the continuation of the study of the following Hamiltonian PDE :
\begin{equation}\label{quad}
i\partial_t u=2J\Pi(|u|^2)+\bar{J}u^2,
\end{equation}
where $u:\mathbb{R}_t\times\mathbb{T}_x\to \mathbb{C}$, the factor $J=J(u):=\int_{\mathbb{T}}|u|^2u\in\mathbb{C}$, and $\Pi$ is the Szeg\H{o} projector onto nonnegative frequencies :
\[ \Pi \left( \sum_{k=-\infty}^{+\infty}u_ke^{ikx}\right)=\sum_{k\geq 0}u_ke^{ikx}.\]
The framework of such an evolution equation was first explored by Gérard and Grellier in their work on the so-called \emph{cubic Szeg\H{o} equation} :
\begin{equation}\label{cubic}
i\partial_t u=\Pi(|u|^2u).
\end{equation}
In a vast series of papers \cite{ann, GGtori, explicit, livrePG}, these authors have shown that equation \eqref{cubic} enjoys very rich dynamical properties, which can be summarized in two words : ``integrability'' and ``turbulence''. Such a discovery paves the way for the study of a various range of other non-dispersive equations, among which the conformal flow on $\mathbb{S}^3$ \cite{Bizon, BizonGround} or the cubic lowest Landau level (LLL) equation \cite{LLL}. In this spirit, our aim would be to understand the behavior of solutions of equation \eqref{quad}, which shares some features with \eqref{cubic} but is far less understood by now.

Let us underline some basic facts about \eqref{quad}. To stick to the notations of \cite{livrePG}, we denote by $L^2_+(\mathbb{T})=\Pi(L^2(\mathbb{T}))$ the Hardy space on the disc (or simply $L^2_+$), and if $G$ is a subspace of $L^2(\mathbb{T})$, $G_+$ will designate the intersection $G\cap L_+^2$. It is well-known that elements of $L_+^2$ can either be considered as Fourier series with only nonnegative modes, or as holomorphic functions $u(z)$ on the open unit disc $\mathbb{D}=\{z\in\mathbb{C}\mid |z|<1\}$ satisfying
\[ \sup_{r\to 1^-} \frac{1}{2\pi}\int_0^{2\pi}|u(re^{i\theta})|^2d\theta <+\infty.\]
The usual scalar product on $L^2$ can be restricted to $L^2_+$, with the convention that for $u,v\in L^2_+$, we have $(u\vert v)=\int_\mathbb{T} u\bar v$ (with respect to the normalized Lebesgue measure $d\theta/2\pi$). Through the standard symplectic form on $L^2_+$ given by $\omega (u,v) := \im (u\vert v)$, we can see \eqref{quad} as the Hamiltonian flow associated to the functional
\[ E(u):=\frac{1}{2}|J|^2= \frac{1}{2}\left| \int_\mathbb{T} |u|^2u \right| ^2.\]
Since the energy $E$ is conserved along flow lines, the modulus of $J$ remains constant, so \eqref{quad} can be regarded as a quadratic system, and we call it therefore the \emph{quadratic Szeg\H{o} equation}. The invariance of $E$ under phase rotation ($u \to e^{-i\theta}u$) and phase translation ($u\to u(ze^{-i\alpha})$) yields two other (formal) conservation laws : the mass $Q(u):=\int_{\mathbb{T}}|u|^2$ and the momentum $M(u):=(-i\partial_xu\vert u)$.

In \cite{thirouin2}, several properties of \eqref{quad} were enlightened, relying on the crucial fact that \eqref{quad} admits a \emph{Lax pair}. Firstly, as in the case of the cubic Szeg\H{o} equation \cite{GKoch}, equation \eqref{quad} admits a well-defined flow on $BMO_+(\mathbb{T})$, where $BMO(\mathbb{T})$ stands for the space of functions with \emph{bounded mean oscillation} introduced by John and Nirenberg. This flow propagates every additional regularity, in the sense that if the initial datum also belongs to $H^s_+(\mathbb{T})$ for some $s>0$, then the corresponding solution will stay in $H^s_+$ for all time. We thus recover a flow on the natural energy space $H^{1/2}_+$ where $E$, $Q$ and $M$ are all well-defined. Secondly, we showed the existence of finite-dimensional submanifolds of $L^2_+$ consisting of rational functions of $z$ which are stable by the flow of \eqref{quad}, and where weak turbulence phenomena can occur, in contrast with what happens for the cubic Szeg\H{o} equation \eqref{cubic}. More precisely, we have proved that there is an invariant manifold of homographies starting of which the solutions of \eqref{quad} in $C^\infty_+(\mathbb{T})$ are bounded in the $H^{1/2}$ norm, but grow exponentially fast in the $H^s$ topology for any $s>1/2$.

The purpose of the present work is to investigate another typical feature of integrable systems : we intend to give a complete classification of all traveling waves that are solutions of equation \eqref{quad} in $H^{1/2}_+$, and to discuss their stability. Recall that a traveling wave is a solution $v(t,z)$ of the Cauchy problem associated to \eqref{quad} for which there exists $\omega, c\in \mathbb{R}$ such that
\begin{equation}\label{trav} v(t,z)=e^{-i\omega t}v_0(ze^{-ict}), \quad \forall t\in\mathbb{R},\ \forall z\in\mathbb{D},
\end{equation}
where $v_0=v(0,\cdot)$ is the initial state. We refer to $\omega$ as the \emph{pulsation}, and to $c$ as the \emph{velocity}.

Hereafter, we only look for such solutions with $c\neq 0$ or $\omega \neq 0$, for the set of steady solutions ($\omega =c=0$) coincides with the subset $\{ J(v)=0\}$ of $H^{1/2}_+$ which seems hard to characterize (see Appendix \ref{steady} for a discussion about that point).

\begin{thm}[Classification of traveling waves]\label{main}
The initial state $v_0(z)\in H^{1/2}_+$ gives rise to a traveling wave solution of \eqref{quad} with $(\omega, c)\neq (0,0)$ if and only if it is a constant, or if there exists $\lambda, p\in\mathbb{C}$ with $0<|p|<1$ and an integer $N\geq 1$ such that one of the following holds :
\begin{enumerate}
\item[\textbf{(i)}] \[ v_0(z)=\frac{\lambda}{1-pz^N},\]
in which case
\[\omega = |\lambda|^4\frac{3-|p|^2}{(1-|p|^2)^3} \quad \text{and} \quad c=\frac{|\lambda|^4}{N}\frac{1}{(1-|p|^2)^2};\]
\item[\textbf{(ii)}] \[ v_0(z)=-\lambda\frac{1+|p|^2}{1-|p|^2}+\frac{\lambda}{1-pz^N},\]
in which case
\[\omega = |\lambda|^4|p|^4\frac{(1+5|p|^2)(3+5|p|^2)}{(1-|p|^2)^4} \quad \text{and} \quad c=-\frac{|\lambda|^4}{N}|p|^4\frac{3+5|p|^2}{(1-|p|^2)^3}.\]
\end{enumerate}
\end{thm}

\begin{rem}
Traveling waves with $c=0$ and $\omega\neq 0$ are called \emph{standing waves}. In particular, Theorem \ref{main} asserts that there are no standing waves for \eqref{quad} in $H^{1/2}_+$ apart from constants. However, if we remove the condition on the regularity of $v_0$ and only ask that $v_0\in BMO_+$, then we will prove that $v_0$ is a standing wave if and only if
\[ v_0=\lambda\cdot\Pi(\mathbbm{1}_\mathcal{B}),\]
where $\lambda\in\mathbb{C}$ and $\mathcal{B}$ is any Borel subset of $\mathbb{T}$.
\end{rem}

Here again, the situation appears to be very different from the one described in \cite[Theorem 1.4]{ann} for the cubic Szeg\H{o} equation. Indeed equation \eqref{cubic} admits many stationary waves in $H^{1/2}_+$ that are given by Blaschke products (see the end of Section \ref{proof_main} for a precise definition). Moreover, it also admits a wider variety of traveling waves, whereas those of Theorem \ref{main} all deduce from one another by an invariance argument (see Proposition \ref{invar}).

The proof of Theorem \ref{main} mainly consists in translating \eqref{trav} into an algebraic relation between operators thanks to the Lax pair. The key of the proof consists in describing the spectrum of a perturbation of the operator $D=\frac{1}{i}\partial_x$ by a bounded Toeplitz operator over $L^2_+$. It is worth noticing that this kind of operators precisely arises in the study of the Lax pair associated to the Benjamin-Ono equation (see \emph{e.g.} \cite{Wu-BO}). See also the pioneering work of Amick and Toland on the solitary waves for the Benjamin-Ono equation \cite{a-toland-2, a-toland}, also dealing with a non-local operator such as $\Pi$.

\vspace{1em}
In the present paper, we would also like to address the question of the nonlinear stability of the traveling wave solutions. Relatively to some norm $\|\cdot\|_X$, we will say that $v_0$ is \emph{$X$-orbitally stable} if for any $\varepsilon >0$, there exists $\eta>0$ such that if $\|u_0-v_0\|_X\leq \eta$, then, for all $t\in\mathbb{R}$,
\[\inf_{(\theta,\alpha)\in\mathbb{T}^2} \|u(t)-e^{-i\theta}v_0(ze^{-i\alpha})\|_X\leq \varepsilon,\]
where $t\mapsto u(t)$ refers to the solution of \eqref{quad} starting from $u_0$ at time $t=0$.

Classically, we have a global Gagliardo-Nirenberg inequality on $H^{1/2}_+$ :
\begin{prop}\label{gagliardo-prop}
For any $u\in H^{1/2}_+$,
\begin{equation}\label{gagliardo}
E(u)\leq \frac{1}{2}Q(u)^2(Q(u)+M(u)).
\end{equation}
Equality in \eqref{gagliardo} holds if and only if $u(z)=\frac{\lambda}{1-pz}$ for some $\lambda,p\in\mathbb{C}$ with $|p|<1$.
\end{prop}

This enables us to prove the following stability result :

\begin{cor}[Stability of the ground states]\label{coro-stab}
For $\lambda,p\in\mathbb{C}$ with $|p|<1$, the traveling wave $z\mapsto \frac{\lambda}{1-pz}$ is $H^{1/2}_+$-orbitally stable.
\end{cor}

Apart from this family of ground states, we are also going to study the stability of the second branch of the traveling waves of Theorem \ref{main}. Using the Lax pair, we know (see Corollary \ref{coro-v}) that the set of functions of the form
\[ z\mapsto b+\frac{cz}{1-pz},\quad  b,c,p\in \mathbb{C},\: c\neq 0,\: c-bp\neq 0, \:|p|<1,\]
is left invariant by the flow of \eqref{quad}, and on this set (which is called $\mathcal{V}(3)$), \eqref{quad} reduces to a system of coupled ODEs on $b$, $c$ and $p$, which have been made explicit in \cite{thirouin2} and is recalled in Appendix \ref{formulaire}. Designing an appropriate perturbation of the ``translated ground states'' inside $\mathcal{V}(3)$, and studying the leading order of the new equations of motion, will yield the following result, which also makes use of the invariance argument mentioned above :

\begin{prop}\label{instab}
For $\lambda,p\in\mathbb{C}$ with $|p|<1$, for $N\in\mathbb{N}$, the traveling waves $z\mapsto -\lambda\frac{1+|p|^2}{1-|p|^2}+\frac{\lambda}{1-pz^N}$ are \emph{not} $H^{1/2}_+$-orbitally stable.
\end{prop}

The question of the stability of traveling waves of the form $\frac{\lambda}{1-pz^N}$, with $N\geq 2$, is left open. In the case of the cubic Szeg\H{o} equation, a negative answer is given in \cite{GGtori} for all traveling waves of degree greater than $2$, but the proof heavily relies on the use of action-angle coordinates which are still missing in our case.

\vspace*{1em}
\begin{rem}[A remark on the traveling waves on $\mathbb{R}$]
As equation \eqref{quad} can be posed on $\mathbb{R}$ as well, the same question of finding traveling waves holds. Adapting rigourously the argument of O. Pocovnicu \cite{PocoSol} leads to the same result as for the cubic Szeg\H{o} equation on the line : the only solitons are given by the profiles
\[ v_0(z)=\frac{\alpha}{z-p},\quad \forall z\in \{\zeta\in\mathbb{C}\mid \im \zeta >0 \},\]
where $\alpha,p\in \mathbb{C}\setminus\{0\}$ with $\im p <0$. Moreover, these solitons are orbitally stable. This remark strengthens an observation that was already made in \cite{thirouin2} : the cubic and quadratic Szeg\H{o} equation look very similar on the line, unlike what happens on the torus. That is why it would be also very interesting to study the stability of the above solitons under a perturbation of the quadratic Szeg\H{o} equation itself, as was done in \cite{PocoToeplitz} for a perturbation of the cubic Szeg\H{o} equation by a Toeplitz potential.
\end{rem}

\vspace*{1em}
This paper is organized as follows. In Section \ref{notations}, we give the definition of the operators arising in the study of equation \eqref{quad} and recall some results about them. In Section \ref{proof_main}, we give the proof of the classification theorem. Section \ref{stab} is devoted to the questions about the stability of the traveling waves that are exhibited in Theorem \ref{main}.

The author is most grateful to Pr. Patrick Gérard for suggesting the study of this problem, and for most valuable comments and remarks during this work.

\vspace*{1em}
\section{Notations and preliminaries}\label{notations}

\subsection{Hankel and Toeplitz operators}
In this paragraph, we recall some notations, the Lax pair result from \cite{thirouin2}, and some of its consequences.

For $u\in BMO_+(\mathbb{T})$, we denote the Hankel operator of symbol $u$ by $H_u: h\mapsto \Pi (u\bar{h})$. This defines a bounded $\mathbb{C}$-antilinear operator on $L^2_+$. Since $\forall h_1,h_2\in L^2_+$, we have $(H_u(h_1)\vert h_2)=(H_u(h_2)\vert h_1)$, we see that $H_u^2$ is $\mathbb{C}$-linear, positive and self-adjoint. Similarly, for $b\in L^\infty(\mathbb{T})$, the Toeplitz operator of symbol $b$ is given by $T_b:h\mapsto \Pi(bh)$. We also have $T_b\in\mathcal{L}(L^2_+)$, but $T_b$ is $\mathbb{C}$-linear. The adjoint of $T_b$ is $(T_b)^*=T_{\bar{b}}$. A fundamental example of a Toeplitz operator is the shift on the right, which we call $S:=T_{e^{ix}}$. Combining these definitions, we can define the shifted Hankel operator of symbol $u\in BMO_+$ : for $h\in L^2_+$, we set
\[ K_u(h):=S^*H_u(h)=H_uS(h)=H_{S^*u}(h).\]
Now we can state the algebraic identities associated to the Lax pair for \eqref{quad}.

\begin{thm}\label{lax_thm}
Let $u\in H^s_+$ for some $s>1/2$. Set $X(u):=2\Pi(|u|^2)+u^2$. Then we have
\begin{gather*}
K_{X(u)} = A_uK_u+K_uA_u \\
H_{X(u)} = A_uH_u+H_uA_u-(u\vert\cdot)u,
\end{gather*}
where $A_u:=T_{u}+T_{\bar{u}}$ is a bounded self-adjoint operator on $L^2_+$.
\end{thm}

\begin{proof}[Proof.]
The proof simply follows from the arguments of \cite{thirouin2}, but we recall it for the convenience of the reader. Taking $h\in L^2_+$, we see that
\[ \Pi (\Pi(|u|^2)\bar{h})=\Pi(|u|^2\bar{h})=\Pi (u \overline{\Pi(uh)})=\Pi( \bar{u}\Pi(u\bar{h})),\]
and on the other hand, decomposing $\Pi (u^2\bar{h})=\Pi(u\Pi(u\bar{h}))+\Pi (u(I-\Pi)(u\bar{h}))$, we observe that
\[ (I-\Pi)(u\bar{h})=\overline{\Pi (\bar{u}h)}-(u\vert h),\]
so $\Pi(u^2\bar{h})=T_uH_u(h)+H_uT_{\bar{u}}(h)-(u\vert h)u$. Now, writing $H_{X(u)}(h)=2J\Pi(\Pi(|u|^2)\bar{h})+\bar{J}\Pi (u^2\bar{h})$ leads to the second identity.

Besides, for $h\in L^2_+$,
\[\begin{aligned} (S^*A_u-A_uS^*)h&=\Pi(\bar{z}\Pi((u+\bar{u})h)-(u+\bar{u})\Pi(\bar{z}h))\\
&=\Pi(\bar{z}uh)-\Pi(u\Pi(\bar{z}h))- \Pi(\bar{z}(I-\Pi)(\bar{u}h))\\
&=(h\vert 1)S^*u,
\end{aligned}\]
because $\bar{z}h=\Pi(\bar{z}h)+\bar{z}(h\vert 1)$, and $\bar{z}(I-\Pi)(\bar{u}h)\perp L^2_+$. Then
\begin{align*}
K_{X(u)}(h)=S^*H_{X(u)}(h)&=A_uK_u(h)+K_uA_u(h)-(u\vert h)S^*u+[S^*,A_u]H_u(h)\\
&=A_uK_u(h)+K_uA_u(h),
\end{align*}
since $(H_u(h)\vert 1)=(H_u(1)\vert h)=(u\vert h)$.
\end{proof}

A consequence of this theorem is the existence of finite-dimensional submanifolds of the energy space, of arbitrary complex dimension, which are stable by the flow of \eqref{quad}, as it is also true for the cubic Szeg\H{o} equation (see \cite{GGtori}).

\begin{déf}
Let $N\in\mathbb{N}$. We denote by $\mathcal{V}(2N)$ the set of functions $u\in H^{1/2}_+(\mathbb{T})$ such that $\rg H_u=\rg K_u=N$, and by $\mathcal{V}(2N+1)$ the set of those such that $\rg H_u=N+1$ and $\rg K_u=N$.
\end{déf}

It is easy to see that $\bigcup_{d\geq 0} \mathcal{V}(d)$ is simply the set of symbols $u$ such that $H_u$ (and $K_u$) is a finite-rank operator. From a theorem by Kronecker, the $\mathcal{V}(d)$'s can be described explicitely :

\begin{prop}[\cite{ann, explicit}]\label{kronecker}
Let $d\in\mathbb{N}$. A function $u(z)\in H^{1/2}_+$ belongs to $\mathcal{V}(d)$ if and only if
\[ u(z)=\frac{A(z)}{B(z)},\quad \forall z\in\mathbb{D},\]
where $A$ and $B$ are two polynomials such that $A\wedge B=1$, $B(0)=1$, $B$ has no root in $\overline{\mathbb{D}}$, and in addition
\begin{itemize}
\item $\deg A \leq N-1$ and $\deg B=N$ if $d=2N$ ;
\item $\deg A=N$ and $\deg B \leq N$ if $d=2N+1$.
\end{itemize}
\end{prop}

In particular, each $\mathcal{V}(d)$ is composed of rational, hence smooth functions. Now we can state the announced consequence of Theorem \ref{lax_thm} :

\begin{cor}\label{coro-v}
For each $d\in \mathbb{N}$, $\mathcal{V}(d)$ is preserved by the flow of \eqref{quad}.
\end{cor}

\begin{proof}[Proof.]
From \cite{thirouin2} we already know that $\mathcal{V}(2N)\cup \mathcal{V}(2N+1)$ is preserved by the flow of \eqref{quad}. Thus it suffices to show that $\mathcal{V}(2N)$ is stable. Let $u_0\in \mathcal{V}(2N)$ and $t\mapsto u(t)$ be the corresponding solution. Since $u_0$ belongs to $C^\infty_+(\mathbb{T})$, so does $u(t)$ for all time $t\in\mathbb{R}$, so we can apply Theorem \ref{lax_thm} and compute, for a given $h\in L^2_+$,
\[ i\frac{d}{dt} H_u(h)=H_{2J\Pi(|u|^2)+\bar{J}u^2}(h)=A_{\bar{J}u}H_u(h)+H_uA_{\bar{J}u}(h)-\bar{J}(u\vert h)u,\]
or in other words, using the $\mathbb{C}$-antilinearity of $H_u$, and the fact that $(u\vert h)=(H_u(h)\vert 1)$,
\[ \frac{d}{dt}H_u=\left(-iA_{\bar{J}u}+\frac{i}{2}\bar{J}(\cdot \vert 1)u \right)H_u +H_u\left(iA_{\bar{J}u} -\frac{i}{2}J(\cdot \vert u)1\right).\]
We give a name to the operator on the left, say $Y_u:=-iA_{\bar{J}u}+\frac{i}{2}\bar{J}(\cdot \vert 1)u$, and note that since $A_{\bar{J}u}$ is self-adjoint, we have in fact
\begin{equation}\label{faux_lax}
\frac{d}{dt}H_u=Y_uH_u+H_uY_u^*.
\end{equation}
This identity is close to a Lax pair, but $Y_u^*\neq -Y_u$. It is still enough to deduce the stability of $\mathcal{V}(2N)$. Indeed, let $V(t)$ be the (global) solution to the following linear Cauchy problem on $\mathcal{L}(L^2_+)$ :
\[\left\lbrace \begin{aligned}
V'(t)&=-Y_{u(t)}^*V(t), \\
V(0)&=I.
\end{aligned} \right. \]
Since $V\in GL(L^2_+)$ at time $t=0$, this remains true for all time. Besides, compute
\[ \frac{d}{dt}(V^*H_uV)=(-V^*Y_u)H_uV+V^*\left( \frac{d}{dt}H_u\right) V+V^*H_u(-Y_u^*V)=0\]
by \eqref{faux_lax}, hence $V^*(t)H_{u(t)}V(t)=H_{u_0}$ for all $t\in\mathbb{R}$. In particular, $\ran H_{u(t)}=(V^*(t))^{-1}\ran H_{u_0}$, so both spaces have the same dimension. This proves that $\forall t\in\mathbb{R}$, $u(t)\in\mathcal{V}(2N)$.
\end{proof}

\begin{rem}
As a weaker version of a Lax pair, an identity such as \eqref{faux_lax} also holds for $H_u^2$. It can be interpreted saying that $H_u^2$ remains equivalent to $H_{u_0}^2$ as a \emph{quadratic form}, whereas $K_u^2$ remains equivalent to $K_{u_0}^2$ as an \emph{operator} on $L^2_+$ (see \cite{thirouin2}).
\end{rem}

\subsection{The spectral theory of \texorpdfstring{$K_u$}{Ku}}\label{spec_theo}
For the sake of completeness, we recall below some properties of $K_u$ which will be useful in the course of the proof of Theorem \ref{main}. For a more general picture, we refer to \cite[Section 3]{livrePG}.

Let $u\in H^{1/2}_+$. Then $H_u^2$, $K_u^2$ are positive compact self-adjoint operators \cite{ann} which satisfy the relation $H_u^2=K_u^2+(\cdot \vert u)u$. For $\rho,\sigma\geq 0$, we denote by
\[ E_u(\rho):=\ker (H_u^2-\rho^2 I),\quad F_u(\sigma):=\ker (K_u^2-\sigma^2 I).\]

\begin{prop}[\cite{livrePG}]
Let $s> 0$ such that $\{0\} \varsubsetneq E_u(s)\cup F_u(s)$. Then one of the following holds true :
\begin{enumerate}
\item $\dim E_u(s)= 1+\dim F_u(s)$, $u\not\perp E_u(s)$ and $F_u(s)=E_u(s)\cap u^\perp$ ($s$ is $H$-dominant) ;
\item $\dim F_u(s)=1+\dim E_u(s)$, $u\not\perp F_u(s)$ and $E_u(s)=F_u(s)\cap u^\perp$ ($s$ is $K$-dominant).
\end{enumerate}
\end{prop}

Denote by $\Sigma^K_u$ the set of $K$-dominant eigenvalues of $K_u^2$ (plus $0$). For $\sigma \in \Sigma^K_u$, let $u_\sigma$ be the projection of $u$ onto $F_u(\sigma)$. By the above proposition, we have $u_\sigma\neq 0$ for $\sigma >0$, and
\begin{equation}\label{reconstr}
u=\sum_{\sigma\in\Sigma^K_u}u_\sigma.
\end{equation}

We will also need a simple lemma :
\begin{lemme}\label{orth1}
Let $\sigma\in \Sigma^K_u\setminus \{0\}$. Then $E_u(\sigma)\subseteq F_u(\sigma)\cap 1^\perp$.
\end{lemme}

\begin{proof}[Proof.]
Suppose $h\in E_u(\sigma)$. Then $H_u(h)\in E_u(\sigma)$, so $H_u(h)\perp u$ by the preceding proposition. Thus $0=(H_u(h)\vert u)=(1\vert H_u^2(h))=\sigma^2(h\vert 1)$.
\end{proof}

\vspace*{1em}
\section{Proof of the main theorem}\label{proof_main}

As usual, we derive from \eqref{trav} an equation on $v_0$ : indeed, if $v$ satisfies both \eqref{quad} and \eqref{trav}, then
\[ i\partial_t v= \omega v+cDv,\]
and $J(v)=e^{-i\omega t}J^0$, where we set $J^0:=J(v_0)=(H_{v_0}^3(1)\vert 1)$. As usual, $D:=-i\partial_x=z\partial_z$. Note that $J^0\neq 0$, since we assume that $u$ is not a steady solution. Thus, we find an equation for the initial data $v_0$ :
\begin{equation}\label{eq-v}
\omega v_0+cDv_0=2J^0\Pi(|v_0|^2)+\overline{J^0}v_0^2.
\end{equation}
Our goal is thus to solve this differential equation. We begin with the necessary conditions and assume that there exists $v_0\in H^{1/2}_+$ such that \eqref{eq-v} holds.

\subsection{The case of standing waves}
Let us first examine the case when $c=0$. To simplify our analysis, we write $v_0=\frac{\omega J^0}{|J^0|^2}u$ so that equation \eqref{eq-v} translates into
\begin{equation}\label{standing}
u=2\Pi(|u|^2)+u^2.
\end{equation}
In order to use Theorem \ref{lax_thm}, we first need to show that the operator $A_u$ can be properly defined. Observe that since $|u|^2$ takes real values, we have
\[ |u|^2=\Pi(|u|^2)+\overline{\Pi(|u|^2)}-Q=\re (u(1-u))-Q.\]
Writing $|u|^2+\re (u^2)=|u|^2+(\re u)^2-(\im u)^2=2(\re u)^2$, we infer that $\re u=r_\pm$ almost everywhere on $\mathbb{T}$, where $r_\pm$ stands for the roots of the polynomial $2X^2-X+Q$. In particular, $\re u\in L^\infty(\mathbb{T})$, which proves that $A_u=T_u+T_{\bar{u}}=2T_{\re u}$ defines a self-adjoint bounded operator on $L^2_+$.

The Lax pair of Theorem \ref{lax_thm} then applies, and yields that
\[ K_u=A_uK_u+K_uA_u.\]
This shows that $A_u$ and $K_u^2$ commute. Now suppose that $K_u^2$ is not zero. As $K_u^2$ is a compact operator (it is even trace class), there exists $V\subseteq L^2_+$ an eigenspace of $K_u^2$ of finite dimension $d\geq 1$. $A_u$ stabilizes $V$, and $A_u\vert_{V}$ is self-adjoint. By the theory of Hermitian operators in finite dimension, $A_u\vert_{V}$ then admits a non-zero eigenvector, \textit{i.e.} there exists $\varphi\in L^2_+\setminus\{0\}$ and $\lambda\in\mathbb{R}$ such that $A_u\varphi=\lambda\varphi$. Hence
\[ 2\Pi (\varphi\re u)=\lambda \varphi,\]
which means that $(2\re u-\lambda)\varphi \perp L^2_+$. Multiplying by $\bar{\varphi}$ yields that $(2\re u-\lambda)|\varphi|^2$ only has negative Fourier modes, but as this last function takes real values, we must have $(2\re u-\lambda)|\varphi|^2\equiv 0$, or equivalently $2\re u \equiv \lambda$ on $\{ \varphi\neq 0\}$. However, a classical result on the Hardy space $L^2_+$ (see \cite[Theorem 17.18]{Rudin}) ensures that since $\varphi$ is not identically zero, $\{ \varphi =0\}$ has zero measure in $\mathbb{T}$. So $2\re u \equiv \lambda$ almost everywhere. Let us write $u=\frac{\lambda}{2}+i\psi$ with $\psi$ a real function. As $\psi=\frac{1}{i}(u-\frac{\lambda}{2})\in L^2_+$, which means that $\psi$ is constant, and so is $u$. This contradicts the assumption that $K_u^2\neq 0$. Hence $K_u^2=0$, and $K_u=0$ (because $\ker K_u=\ker K_u^2$), and finally $S^*u=0$, which means that $u$ (therefore $v_0$) is constant.

Conversely, the set of constant functions corresponds to the manifold $\mathcal{V}(1)=\{ u\in H^{1/2}_+\mid K_u= 0\}$ which is preserved by the flow. On this manifold, \eqref{quad} becomes the following ODE : $if'=3|f|^4f$. Its solutions satisfy $|f|^2= \mathrm{cst}$, and so turn out to be standing waves.

\vspace{1em}
\begin{rem}
However, it appears from the previous analysis that equation \eqref{standing} admits many non-trivial solutions on $BMO_+(\mathbb{T})$, and we can also classify them. Indeed, pick some real number $0<r_-<\frac{1}{6}$, and set $r_+:=\frac{1}{2}-r_-$. Define
\[ \theta :=\frac{r_-}{r_+-r_-}=\frac{2r_-}{1-4r_-}\in (0,1),\]
and let $\mathcal{B}_+\subseteq\mathbb{T}$ be \emph{any} Borel set of measure $\theta$. Let $\mathcal{B}_-:=\mathbb{T}\setminus\mathcal{B}_+$. Lastly, set $f:\mathbb{T}\to\mathbb{R}$, with $f=r_+\mathbbm{1}_{\mathcal{B}_+}+r_-\mathbbm{1}_{\mathcal{B}_-}$, and introduce
\[ u(z):=\frac{1}{2\pi}\int_0^{2\pi}\frac{e^{ix}+z}{e^{ix}-z}f(e^{ix})dx,\quad \forall z\in\mathbb{D}.\]

With these definitions, we notice that $f\in L^2(\mathbb{T})$, and thanks to the Poisson kernel, we see that $\re u$ equals $r_+$ (resp. $r_-$) on $\mathcal{B}_+$ (resp. $\mathcal{B}_-$). It also appears that $u$ is holomorphic on $\mathbb{D}$, and expanding $(1-ze^{-ix})^{-1}$ as a power series, it can be checked that $u=2\Pi(f)-(f\vert 1)$, so $u\in \Pi(L^\infty)=BMO_+(\mathbb{T})$.

As above, we have $2(\re u)^2-\re u +2r_-r_+=0$, which implies that
\[ |u|^2=\re (u(1-u))-2r_-r_+.\]
Applying the Szeg\H{o} projector, we get $\Pi(|u|^2)=\frac{1}{2}u(1-u)+\frac{1}{2}(1\vert u)(1-(1\vert u))-2r_-r_+$. Furthermore, $(u\vert 1)=(f\vert 1)=r_+\theta+r_-(1-\theta)r_-=2r_-$, so we compute
\[\frac{1}{2}(1\vert u)(1-(1\vert u))-2r_-r_+=r_-(1-2r_--2r_+)=0.\]
This proves that $u$ solves \eqref{standing}. Note that we have
\[ u=\frac{1}{1+2\theta}\cdot\Pi(\mathbbm{1}_{\mathcal{B}_+}).\]
\end{rem}

\subsection{The general case}

\paragraph{Traveling waves are rational functions.} From now on, we assume that $c\neq 0$. Observe that from \eqref{eq-v}, it implies that $v_0\in C^\infty_+(\mathbb{T})$. If we make the ansatz $v_0=\frac{cJ^0}{|J^0|^2}u$, then equation \eqref{eq-v} reduces to the following equation on the profile $u$ :
\begin{equation}\label{eq-u}
\varpi u+Du=2\Pi(|u|^2)+u^2,
\end{equation}
where we have set $\varpi :=\omega/c$.

Writing $K_{Du}=DK_u+K_uD+K_u$ and $H_{Du}=DH_u+H_uD$, the Lax pair identities from Theorem \ref{lax_thm} read
\begin{align}
(\varpi +1)K_u&=(A_u-D)K_u+K_u(A_u-D),\label{commutK} \\
\varpi H_u&=(A_u-D)H_u+H_u(A_u-D)-(u\vert \cdot)u.\label{commutH}
\end{align}
Now, since $D$ is a selfadjoint operator with compact resolvent and $A_u$ is a selfadjoint bounded operator, $D-A_u$ is also a selfadjoint operator with compact resolvent, so we can find an eigenbasis $\{\varepsilon_j\}_{j\in\mathbb{N}}$ of $L^2_+$ and real eigenvalues $\lambda_j$ going to $+\infty$, such that $(D-A_u)\varepsilon_j=\lambda_j\varepsilon_j$, $\forall j\in\mathbb{N}$. Plugging this into \eqref{commutK}, we find
\[ (-\varpi -1-\lambda_j)K_u(\varepsilon_j)=(D-A_u)K_u(\varepsilon_j),\]
so $-\varpi-1-\lambda_j$ is also an eigenvalue of $D-A_u$, unless $K_u(\varepsilon_j)=0$. Since $((D-A_u)h\vert h)\geq -M\|h\|^2$ for some $M>0$, we deduce that $K_u(\varepsilon_j)=0$ for $j$ large enough, hence $K_u$ has finite rank. This proves that $u$ is a rational function by Proposition \ref{kronecker}.

We are now going to deduce further informations on $u$ thanks to equation \eqref{eq-u} itself. Indeed expand $u(z)$ as a linear combination of
\[ z^k, \ 0\leq k\leq m_0, \quad \text{and} \quad \frac{1}{(1-p_\ell z)^k}, \ 1\leq k\leq m_\ell, \]
where $\mathcal{P}=\{ p_\ell \ | \ 1\leq \ell \leq N\}$ is some finite set of points of $\mathbb{D}\setminus \{0\}$, and the $m_\ell$'s are nonnegative integers. Now $\Pi(|u|^2)=H_u(u)\in \im H_u$, so it is a combination of the same fractions (see \cite[appendix 4]{ann} for a complete description of $\im H_v$ when $v$ is a rational function). As $D=z\partial_z$, $Du$ is a linear combination of $z^k$, $0\leq k\leq m_0$, and of $(1-p_\ell z)^{-k}$, $1\leq k\leq m_\ell+1$, $p_\ell\in\mathcal{P}$. However, the $u^2$ term in \eqref{eq-u} generates terms like $z^{2m_0}$ and $(1-p_\ell z)^{-2m_\ell}$, and they cannot be compensated by other terms unless $m_0=0$, and $m_\ell=1$ for all $\ell$. In other words, $u$ must have simple poles, and we write
\[ u=\beta +\sum_{\ell =1}^N \frac{\alpha_\ell}{1-p_\ell z},\]
for some $\beta$, $\alpha_1,\ldots ,\alpha_N\in\mathbb{C}$.

Let us compute
\begin{align*}
\Pi(|u|^2) &= \sum_{\ell=1}^N \left( \bar{\beta}+\sum_{\kappa=1}^N \frac{\alpha_\ell\overline{\alpha_\kappa}}{1-p_\ell\overline{p_\kappa}} \right) \frac{1}{1-p_\ell z}+|\beta|^2+\sum_{\ell=1}^N \overline{\alpha_\ell}\beta ,\\
u^2 &= \sum_{\ell=1}^N\frac{\alpha_\ell^2}{(1-p_\ell z)^2}+\sum_{\ell=1}^N \left(2\beta\alpha_\ell+ \sum_{\kappa\neq \ell}2p_\ell\frac{\alpha_\ell\alpha_\kappa}{p_\ell-p_\kappa}\right) \frac{1}{1-p_\ell z}+\beta^2 ,\\
Du&= \sum_{\ell=1}^N\frac{\alpha_\ell}{(1-p_\ell z)^2}-\sum_{\ell=1}^N \frac{\alpha_\ell}{1-p_\ell z}.
\end{align*}
Considering the multiples of $\frac{1}{(1-p_\ell z)^2}$ which only appear when computing $Du$ and $u^2$, we get that $\alpha_\ell=1$, for all $1\leq \ell\leq N$ (assuming without loss of generality that $\alpha_\ell\neq 0$).

\paragraph{Reduction to the case $u\in \mathcal{V}(2N)$.} From the above formulae on $u$, $\Pi(|u|^2)$ and $u^2$, we also get an equation on $\beta$ which reads
\[ \varpi \beta = 2|\beta|^2+2N\beta +\beta^2.\]
This equation only has two solutions : either $\beta=0$ or, dividing by $\beta$, $2\bar{\beta}+\beta=\varpi -2N$. This shows that $\beta=\varpi-2N-2\re (\beta)\in \mathbb{R}$, so we must have
\begin{equation}\label{beta}
\beta =\frac{1}{3}(\varpi -2N).
\end{equation}

Now, introduce $\tilde{u}=u-\beta$. Then
\begin{align*}
2\Pi(|\tilde{u}|^2)+\tilde{u}^2&=2\Pi(|u|^2)+2|\beta|^2-2\beta\Pi(\bar{u})-2\bar{\beta}u+u^2-2u\beta+\beta^2 \\
&=2\Pi(|u|^2)+u^2-4\beta u+3\beta^2-2\beta (1\vert u) \\
&=(\varpi-4\beta) u+Du +\beta^2-2\beta N\\
&=(\varpi-4\beta) \tilde{u}+D\tilde{u}.
\end{align*}

Consequently, up to a modification of $\varpi$, it suffices to treat the case $\beta=0$, so from now on, we consider that $u$ belongs to $\mathcal{V}(2N)$ and is of the form
\begin{equation}\label{DES}
u=\sum_{\ell=1}^N \frac{1}{1-p_\ell z}.
\end{equation}
In particular, we have
\begin{equation}\label{moyenne}
(u\vert 1)=N.
\end{equation}

To conclude, it suffices to describe the possible choices of $\mathcal{P}=\{p_\ell\}$. Note that the $p_\ell$'s solve the following system of equations :
\begin{equation}\label{syst_pl}
\frac{\varpi -1}{2}=\sum_{\kappa=1}^N \frac{1}{1-p_\ell\overline{p_\kappa}}+\sum^N_{\substack{\kappa=1 \\ \kappa\neq \ell}}\frac{p_\ell}{p_\ell-p_\kappa},\quad \forall \ell =1,\ldots, N.
\end{equation}

\paragraph{Spectral analysis of $A_u-D$.} To solve the above system, we are going to perform a spectral analysis of the operator $A_u-D$, taking advantage from its relation with the self-adjoint finite-rank operator $K_u^2$, in the spirit of \cite[Section 9]{ann}. Denote by $\Sigma$ the (finite) set of the $K$-dominant eigenvalues of $K_u^2$. In this whole section, we fix $\sigma\in \Sigma$. To simplify the notations, we set $F:=F_u(\sigma)$ and $E:=E_u(\sigma)$ (see paragraph \ref{spec_theo}). Since $0\notin \Sigma$ now that $u\in\mathcal{V}(2N)$, both $F$ and $E$ are finite-dimensional subspaces of $\ran H_u=\ran K_u$. We also define $u_\sigma$ to be the (non-zero) orthogonal projection of $u$ onto $F$. Our goal is to prove the following statement :

\begin{prop}\label{Au-D}
We have
\[ (A_u-D)u_\sigma=\frac{\varpi +n_\sigma}{2}u_\sigma,\]
where $n_\sigma:=\dim F$.
\end{prop}

\begin{proof}[Proof.] The proof of this proposition decomposes into several steps.

\vspace*{1ex}
\noindent\underline{First step.}\hspace*{1ex} $u_\sigma$ is an eigenvector of $A_u-D$.

From \eqref{commutK}, we see that $(A_u-D)K_u^2=K_u^2(A_u-D)$. This shows that $(A_u-D)(F)\subseteq F$. Now, as $E=F\cap u^\perp$, then by \eqref{commutH}, $\varpi H_u=(A_u-D)H_u+H_u(A_u-D)$ on $E$.

If $h\in E$, then $H_u(h)\in E$, so we apply this identity to $H_u(h)$. We get
\[\begin{aligned} \sigma^2\varpi h&=\sigma^2(A_u-D)(h)+H_u(A_u-D)H_u(h)\\
&=\sigma^2(A_u-D)(h)+H_u( -H_u(A_u-D)+\varpi H_u)(h),\end{aligned}\]
so $H_u^2(A_u-D)(h)=\sigma^2(A_u-D)(h)$, which proves that $A_u-D$ also leaves $E$ invariant. As $A_u-D$ is self-adjoint, it thus preserves $F\cap E^\perp = \mathbb{C}u_\sigma$, which means that there exists $\lambda\in\mathbb{R}$ such that
\[(A_u-D)u_\sigma=\lambda u_\sigma.\]

\vspace*{1ex}
\noindent\underline{Second step.}\hspace*{1ex} $2\lambda-\varpi$ is a positive integer.

Assume by contradiction that it is not true. Then by \eqref{commutK}, $K_uu_\sigma$ satisfies $(A_u-D)K_uu_\sigma=(\varpi +1 -\lambda)K_uu_\sigma$. As $\varpi+1-\lambda\neq\lambda$ by assumption, we find that $K_uu_\sigma\perp u_\sigma$, and since $K_uu_\sigma\in F$, $K_uu_\sigma\perp u$, hence $K_uu_\sigma\in E$. Applying identity \eqref{commutH}, we get
\[(A_u-D)H_uK_uu_\sigma=(\lambda-1)H_uK_uu_\sigma.\]
Note that $H_uK_uu_\sigma\neq 0$, for $K_u$ (resp. $H_u$) is one-to-one on $F$ (resp. $E$).

As $H_uK_uu_\sigma\in E$ and also in $F$, we can restart this argument, and prove by induction that
\[ (A_u-D)(H_uK_u)^ju_\sigma=(\lambda -j)(H_uK_u)^ju_\sigma,\quad \forall j\in\mathbb{N}.\]
This of course cannot happen, otherwise $A_u-D$ would have infinitely many distinct eigenvalues on $F$ which is a finite dimensional subspace of $L^2_+$. 

From now on, we write $\lambda = \frac{1}{2}(\varpi +n)$. It remains to show that $n=\dim F$.

\vspace*{1ex}
\noindent\underline{Third step.}\hspace*{1ex} The action of $S^*$ on $E$.

Our purpose now is to prove that if $e\in E\setminus\{0\}$ satisfies $(A_u-D)e=\mu e$, then $S^*e\in F\setminus\{0\}$ and satisfies $(A_u-D)S^*e=(\mu +1)S^*e$.

We thus have to compute the commutator $[A_u-D,S^*]$. Since for $h\in L^2_+$,
\[[D,S](h)=z\partial_z(zh)-z^2\partial_zh=zh=S(h),\]
we get $[S^*,D]=([D,S])^*=S^*$. Besides, we know from the proof of Theorem \ref{lax_thm} that $S^*A_u-A_uS^*=(\cdot\vert 1)S^*u$. But $E\subseteq 1^\perp$ by Lemma \ref{orth1}. This proves that $A_u$ and $S^*$ commute on $E$.

Finally, taking $e\in E$ as above, we get
\[ (A_u-D)S^*e=S^*(A_u-D)e+[A_u-D,S^*]e=\mu S^*e+[S^*,D]e=(\mu +1)S^*e.\]
To see that $S^*e\in F\setminus\{0\}$, it suffices to notice that $\sigma^2S^*e=S^*H_u^2e=K_uH_ue$, and to conclude thanks to the injectivity of $H_u$ and $K_u$ as in the second step.

\vspace*{1ex}
\noindent\underline{Fourth step.}\hspace*{1ex} The eigenvalue $\lambda$ of $A_u-D\vert_F$ is simple.

Otherwise, there would be an eigenvector $e\in F$ associated to $\lambda$ such that $e\perp u_\sigma$. This would mean that $e\in E$, so by the previous point, $S^*e$ would be an eigenvector of $A_u-D\vert_F$ associated to $\lambda +1$, hence orthogonal to $u_\sigma$. Therefore, for all $j\geq 0$, $(S^*)^je$ would be a (non-zero) eigenvector of $A_u-D\vert_F$ associated to $\lambda +j$, and this contradicts the fact that $\dim F<\infty$.

Similarly, $A_u-D\vert_F$ has no eigenvalue $\mu >\lambda$, otherwise, iterating the third step, $\mu+ j$ would be an eigenvalue of $A_u-D\vert_F$ for all $j\in\mathbb{N}$.

Now we are able to finish the proof. Let $\mu$ be the smallest eigenvalue of $A_u-D$ on $F$ and $\tilde{e}$ a corresponding eigenvector. Then $\nu:=\lambda-\mu$ should be a nonnegative integer, or then the $(S^*)^j\tilde{e}$, $j\geq 0$, would give rise to an infinite sequence of orthogonal eigenvectors. We also have $(S^*)^{\nu}\tilde{e}\in\mathbb{C}u_\sigma$, thanks to the previous remark. Moreover, since $S^*$ is injective on $E$ (recall that $E\subseteq 1^\perp$), all the eigenvalues $\mu ,\mu +1, \ldots, \mu+\nu -1$ are simple as well on $F$. Thus $\dim F=\nu +1$.

It is now straightforward to see that $\nu =n-1$, with $n=2\lambda-\varpi$ as above. On the one hand, $K_u(u_\sigma)$ is an eigenvector of $A_u-D\vert_F$ associated to $\varpi +1-\lambda$, which means that $\nu=\lambda-\mu \geq \lambda - (\varpi +1-\lambda)=n-1$ by minimality of $\mu$. On the other hand, $K_u(\tilde{e})$ is an eigenvector associated to $\varpi +1-\mu$, so $\nu=\lambda-\mu \geq (\varpi+1-\mu )-\mu= \varpi+1-2(\lambda -\nu)$ by maximality of $\lambda$, and hence $n-1\geq \nu$. This establishes the yielded formula :
\[ (A_u-D)u_\sigma = \frac{\varpi +\dim F}{2}u_\sigma,\]
in addition to the fact that $\varpi+1-\lambda=\frac{1}{2}(\varpi +2-\dim F)$ is the smallest eigenvalue of $A_u-D\vert_F$.
\end{proof}

We point out a by-product of the last part of the proof :
\begin{cor}\label{angle-Ku}
There exists a complex number $\zeta_\sigma\in\mathbb{C}\setminus\{0\}$ such that we have
\begin{equation}\label{angle-Ku-formule} K_u(u_\sigma)= \zeta_\sigma z^{n_\sigma -1}u_\sigma,
\end{equation}
where $n_\sigma=\dim F$ as above.
\end{cor}

\begin{proof}[Proof.]
Indeed, with the terminology of the proof above, $K_u(u_\sigma)$ is non-zero and colinear to $\tilde{e}$. So $(S^*)^{n_\sigma-1}K_u(u_\sigma)=\zeta_\sigma u_\sigma$ for some $\zeta_\sigma\neq 0$. Now if $j< n_\sigma-1$, then $(S^*)^jK_u(u_\sigma)\in E$, and on $E$, we have $SS^*=I-(\cdot\vert 1)1=I$. Therefore, $S^{n_\sigma-1}(S^*)^{n_\sigma-1}K_u(u_\sigma)=K_u(u_\sigma)$, and \eqref{angle-Ku-formule} is proved.
\end{proof}

We summarize the results of this paragraph on Figure \ref{dessin}.
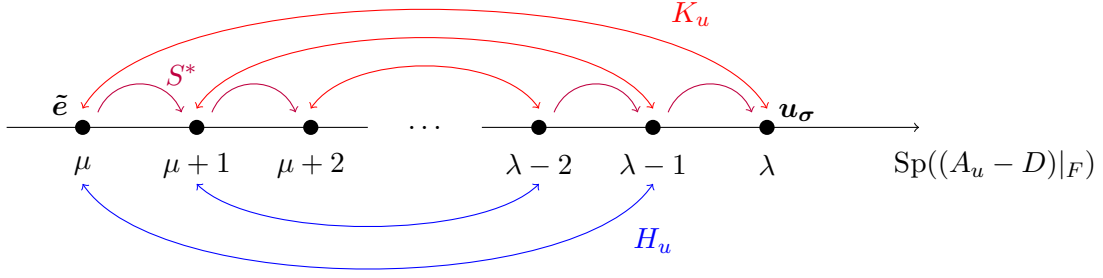
\begin{figure}[ht!]
\begin{center}
\begin{tikzpicture}
\shorthandoff{:;!?};
\fill (1,0) circle (0.1cm);
\fill (2.5,0) circle (0.1cm);
\fill (4,0) circle (0.1cm);
\fill (7,0) circle (0.1cm);
\fill (8.5,0) circle (0.1cm);
\fill (10,0) circle (0.1cm);

\draw (0,0) node{} -- (4.75,0) node{};
\draw[->] (6.25,0) node{} -- (12,0) node{};
\draw (5.5,0) node{$\ldots$};

\draw (1,-0.5) node {$\mu$};
\draw (2.5,-0.5) node {$\mu +1$};
\draw (4,-0.5) node {$\mu +2$};
\draw (7,-0.5) node {$\lambda -2$};
\draw (8.5,-0.5) node {$\lambda -1$};
\draw (10,-0.5) node {$\lambda$};
\draw (13,-0.5) node {$\mathrm{Sp}((A_u-D)\vert_F)$};
\draw (0.7,0.3) node {\textbf{$\tilde{e}$}};
\draw (10.4,0.2) node {\textbf{$u_\sigma$}};

\draw[red,<->] (1,0.25) .. controls (2,2) and (9,2) .. (10,0.25);
\draw[red,<->] (2.5,0.25) .. controls (3.25,1.5) and (7.75,1.5) .. (8.5,0.25);
\draw[red,<->] (4,0.25) .. controls (4.5,1) and (6.5,1) .. (7,0.25);
\draw[blue,<->] (1,-0.75) .. controls (2,-2.25) and (7.5,-2.25) .. (8.5,-0.75);
\draw[blue,<->] (2.5,-0.75) .. controls (3.25,-1.5) and (6.25,-1.5) .. (7,-0.75);
\draw[purple,->] (1.2,0.2) .. controls (1.4,0.7) and (2.1,0.7) .. (2.3,0.2);
\draw[purple,->] (2.7,0.2) .. controls (2.9,0.7) and (3.6,0.7) .. (3.8,0.2);
\draw[purple,->] (7.2,0.2) .. controls (7.4,0.7) and (8.1,0.7) .. (8.3,0.2);
\draw[purple,->] (8.7,0.2) .. controls (8.9,0.7) and (9.6,0.7) .. (9.8,0.2);

\node[text=red] at (9,1.5) {$K_u$};
\node[text=blue] at (8.5,-1.5) {$H_u$};
\node[text=purple] at (2.3,0.7) {$S^*$};

\end{tikzpicture}
\end{center}
\caption{The action of $H_u$, $K_u$ and $S^*$ on $F$.}
\label{dessin}
\end{figure}

\paragraph{$K_u^2$ has at most one positive eigenvalue.}
Combining the informations of Proposition \ref{Au-D} on each $K$-dominant eigenspace of $K_u^2$, we claim that $K_u$ cannot have more than one singular value. Let $\Sigma = \{ \sigma_m \mid 1\leq m\leq |\Sigma |\}$ be the set of the $K$-dominant eigenvalues of $K_u^2$, and $n_m=\dim (F_u(\sigma_m))$ for each $m$. As above, we introduce the orthogonal projection of $u$ onto the eigenspace $F_u(\sigma_m)$ and call it $u_m$. The norm $\|\cdot\|$ without further precision denotes the $L^2$ norm, and $Q$ refers to $\|u\|^2$.

We first prove a simple but crucial lemma.

\begin{lemme}
We have, for each $1\leq m\leq |\Sigma |$,
\begin{equation}\label{umvm}
(\varpi +n_m-2N)\cdot (u_m\vert 1)=2\|u_m\|^2.
\end{equation}
In particular, $(u_m\vert 1)$ is a positive real number.
\end{lemme}

\begin{proof}[Proof.]
We know that $u_m$ is a non-zero eigenvector of the operator $A_u-D$. Let $\lambda$ be the associated eigenvalue. Since $(A_u-D)1=\Pi(u+\bar{u})$, and by \eqref{moyenne}, $\Pi(\bar{u})=(1\vert u)=N$, we can compute
\[ \lambda (u_m\vert 1)=((A_u-D)u_m\vert 1)=(u_m\vert (A_u-D)1)=\|u_m\|^2+N(u_m\vert 1),\]
so that $(u_m\vert 1)(\lambda -N)=\|u_m\|^2$. But by Proposition \ref{Au-D}, we have $\lambda=\frac{1}{2}(\varpi +n_m)$, which gives formula \eqref{umvm}.

As $\|u_m\|^2>0$, it remains to show that $\varpi+n_m-2N$ is a positive number. Taking the scalar product of equation \eqref{eq-u} with $1$ leads to the following identity :
\begin{equation}\label{Q}
\varpi N=2Q+N^2.
\end{equation}
However, the Cauchy-Schwarz inequality yields $N^2=|(u\vert 1)|^2\leq Q$, so $\varpi N\geq 3N^2$, which means that $\varpi \geq 3N$. A fortiori $\varpi+n_m-2N\geq \varpi -2N>0$, and the proof of the lemma is complete.
\end{proof}

We are now able to prove our claim :
\begin{lemme}\label{singleton}
Necessarily, $\Sigma$ is a singleton $\{\sigma\}$, and $\dim F_u(\sigma)=N$.
\end{lemme}
\begin{proof}[Proof.]
We are going to sum identity \eqref{umvm} over $m$. This gives, using \eqref{reconstr} and the fact that $u_m\perp u_{m'}$ for $m\neq m'$,
\[ (\varpi-2N)\cdot (u\vert 1)+ \sum_{m=1}^{|\Sigma|}n_m(u_m\vert 1)=2\|u\|^2.\]
As we know from \eqref{Q} that $2Q=\varpi N-N^2=(\varpi -N)\cdot (u\vert 1)$, we get
\[ \sum_{m=1}^{|\Sigma|}n_m(u_m\vert 1)=N \cdot (u\vert 1)=N\sum_{m=1}^{|\Sigma|} (u_m\vert 1),\]
or equivalently
\begin{equation}\label{final_sum}
\sum_{m=1}^{|\Sigma|} (N-n_m)\cdot (u_m\vert 1)=0.
\end{equation}

But $\sum_mn_m\leq \rg K_u^2=N$, so $n_m\leq N$. Together with the preceding lemma, this shows that the sum in \eqref{final_sum} only involves nonnegative terms. Each of them must then be zero, \emph{i.e.} we must have $n_m=N$ for all $1\leq m\leq {|\Sigma|}$. Hence $N\geq \sum_mn_m=N|\Sigma|$, so finally $|\Sigma|=1$.
\end{proof}

\paragraph{Determination of $u$.}
Because of \eqref{reconstr}, Lemma \ref{singleton} implies that $u$ is an eigenvector of $K_u^2$, \emph{i.e.} $K_u^2(u)=\sigma^2 u$ for some $\sigma >0$. Hence $H_u^2(u)=(\sigma^2+Q)u$, so $u$ is also an eigenvector of $H_u^2$. Now, coming back to Proposition \ref{Au-D}, this means that $u_\sigma=u$, and $\dim F=N$, hence
\[ (A_u-D)u=\frac{\varpi +N}{2}u.\]
However, $(A_u-D)u=\Pi(|u|^2)+u^2-Du=\varpi u-\Pi(|u|^2)$, by equation \eqref{eq-u}. Therefore,
\[H_u(u)=\varpi u-(A_u-D)u=\frac{\varpi -N}{2}u.\]
On the other hand, by Corollary \ref{angle-Ku}, we also have, for some $\zeta\neq 0$,
\[K_u(u)=\zeta z^{N-1}u.\]
Combining the two informations, we can write $\zeta z^{N-1}u=K_u(u)=S^*H_u(u)=\frac{\varpi -N}{2}S^*u$. Applying $S$ again, we get the existence of some $\alpha\neq 0$ such that $SS^*u=\alpha z^Nu$. Besides, $SS^*u=u-(u\vert 1)=u-N$. This finishes to prove that
\[ u(z)=\frac{N}{1-\alpha z^N}.\]
In other words, going back to formula \eqref{DES}, this shows that $\mathcal{P}$ in \eqref{syst_pl} has to be the set of the $N$th roots of some complex number $\alpha$ such that $|\alpha|<1$.

\begin{rem}
In terms of the inverse spectral transform of \cite[Section 5]{livrePG}, we have proved that $H_u$ has one dominant singular value of multiplicity $1$, and that $K_u$ has one dominant singular value $\sigma$ of multiplicity $N$. Corollary \ref{angle-Ku} says that the angle associated to $\sigma$ is $e^{i\psi}z^{N-1}$ for some $\psi\in\mathbb{T}$.
\end{rem}

\paragraph{The reverse statement.} Now that we have found what form a traveling wave should have a priori, we need to show that functions of this form are traveling waves for \eqref{quad} indeed. From a straightforward computation based on \eqref{syst_pl}, we can easily find the traveling waves of $\mathcal{V}(2)$ and $\mathcal{V}(3)$ : the function $u(z)=\frac{1}{1-pz}$, for $|p|<1$, satisfies
\[ \varpi u+Du=2\Pi(|u|^2)+u^2, \quad\text{with }\varpi \in\mathbb{R},\]
and no other function of $\mathcal{V}(2)$ does. Besides, $\varpi=\frac{3-|p|^2}{1-|p|^2}$ because of \eqref{syst_pl}. Then, if $\lambda\in\mathbb{C}$, set $v:=\lambda u$. We have
\[J(v)=|\lambda|^2\lambda J(u)=\frac{|\lambda|^2\lambda}{(1-|p|^2)^2},\]
which leads to
\[ 2J(v)\Pi(|v|^2)+\overline{J(v)}v^2=\frac{|\lambda|^4\lambda}{(1-|p|^2)^2}(2\Pi(|u|^2)+u^2)=\frac{|\lambda|^4(3-|p|^2)}{(1-|p|^2)^3}v+\frac{|\lambda|^4}{(1-|p|^2)^2}Dv.\]
Thus $v$ gives rise to a traveling wave solution of \eqref{quad} with
\[ \omega =|\lambda|^4\frac{3-|p|^2}{(1-|p|^2)^3},\quad c=|\lambda|^4\frac{1}{(1-|p|^2)^2}. \]

Similarly, in $\mathcal{V}(3)$, the only solution of $\varpi u+Du=2\Pi(|u|^2)+u^2$ has the form $u=\beta +\frac{1}{1-pz}$, and by \eqref{beta} and \eqref{syst_pl}, we must have
\[\left\lbrace \begin{aligned}
\varpi -4\beta&=\frac{3-|p|^2}{1-|p|^2},\\
\varpi -3\beta&=2,
\end{aligned}\right.\]
or equivalently $\beta = -\frac{1+|p|^2}{1-|p|^2}$ and $\varpi=-\frac{1+5|p|^2}{1-|p|^2}$. Then $v:=\lambda u$ for $\lambda\in \mathbb{C}$ is such that
\[ J(v)=|\lambda|^2\lambda J(u)=-\frac{|\lambda|^2\lambda\cdot |p|^4(3+5|p|^2)}{(1-|p|^2)^3},\]
as show the formulae of Appendix \ref{formulaire} for instance. As above, we finally get that $v$ is the initial state of a traveling wave with
\[ \omega =|\lambda|^4\frac{|p|^4(3+5|p|^2)(1+5|p|^2)}{(1-|p|^2)^4},\quad c=-|\lambda|^4\frac{|p|^4(3+5|p|^2)}{(1-|p|^2)^3}.\]

The general case of $\mathcal{V}(2N)$ and $\mathcal{V}(2N+1)$ relies on an invariance argument discovered in \cite{Xu3} in the case of the cubic Szeg\H{o} equation, and which also applies here.

\begin{déf}
An \emph{inner function} is a function $f\in L^2_+$ such that for almost every $x\in\mathbb{T}$,
\[|f(e^{ix})|=1.\]
\end{déf}

\noindent Among such functions, some are rational functions : they are called \emph{Blaschke products}, and are of the form
\[ z\mapsto \prod_{\ell=1}^L \frac{z-\overline{p_\ell}}{1-p_\ell z},\quad p_\ell\in\mathbb{D},\, \forall 1\leq \ell\leq L.\] 

The invariance result is the following :

\begin{prop}[Invariance by composition with an inner function]\label{invar}
Let $f\in L^2_+$ be an inner function, and $t\mapsto u(t,z)$ a solution of \eqref{quad} starting from $u(0,z)=u_0(z)\in BMO_+(\mathbb{T})$. Then
\[ t\mapsto u(t,zf(z))\]
is the solution of \eqref{quad} in $BMO_+$ starting from initial data $u_0(zf(z))$.
\end{prop}

Let us first show how this proposition allows us to conclude. Assume $v(t,z)$ is a traveling wave with $v(0,z)=v_0(z)$. Equation \eqref{trav} is equivalent to
\[ \widehat{v(t)}(k)=\widehat{v_0}(k)e^{-i(\omega +c\cdot k)t}, \quad\forall k\in\mathbb{Z},\]
where $\,\widehat{\cdot}\,$ is the Fourier transform. Now, if $N\geq 1$ and $f(z)=z^{N-1}$, the invariance proposition states that $v(t,z^N)$ is the solution starting from $v_0(z^N)$. Writing $w(t,z):=v(t,z^N)$, we get
\[ \widehat{w(t)}(k)=\widehat{v_0}(k)e^{-i(\omega +\frac{c}{N}\cdot k)t}, \quad\forall k\in\mathbb{Z}\text{ s.t. }N\vert k.\]
Thus, the set of traveling waves for equation \eqref{quad} is stable by composition with $z^N$. Moreover, the pulsation does not change and the velocity is divided by $N$. This ends the proof of Theorem \ref{main}.

\vspace*{1ex}
To prove Proposition \ref{invar}, we introduce an operator on $L^2(\mathbb{T})$. If $f$ is an inner function, we set
\[ \mathcal{T}_f: \left\lbrace \begin{matrix}
L^2(\mathbb{T})&\longrightarrow &L^2(\mathbb{T}) \\
\sum_{k=-\infty}^{+\infty} u_ke^{ikx}&\longmapsto &\sum_{k=-\infty}^{+\infty}u_k[e^{ix}f(e^{ix})]^k.
\end{matrix}\right. \]
The operator $\mathcal{T}_f$ satisfies some useful algebraic properties that we sum up in the next lemma.

\pagebreak
\begin{lemme}\label{Tf_alg}
\begin{enumerate}
\item $\mathcal{T}_f$ is an isometry of $L^2(\mathbb{T})$.
\item $\mathcal{T}_f\circ \Pi=\Pi\circ \mathcal{T}_f$.
\item $\mathcal{T}_f$ induces a $*$-endomorphism of the C$^*$-algebra $L^\infty(\mathbb{T})$.
\end{enumerate}
\end{lemme}

\begin{proof}[Proof.]
For $k\in\mathbb{Z}$, let us introduce $e^f_k(x)=e^{ikx}f(e^{ix})^k$ for $x\in\mathbb{T}$. To see that $\mathcal{T}_f$ is an isometry, it suffices to notice that $\{e^f_k\}_{k\in \mathbb{Z}}$ is a orthonormal family of $L^2$. If $p\geq q$, we have
\[ (e^f_p\vert e^f_q)=\frac{1}{2\pi}\int_{-\pi}^{\pi} e^f_p(x)\overline{e^f_q(x)}dx=\frac{1}{2\pi}\int_{-\pi}^{\pi}e^{i(p-q)x}f(e^{ix})^{p-q}dx=\delta_{pq},\]
for $f$ has nonnegative frequencies only. Hence, for any $u=\sum_{k=-\infty}^{+\infty} u_ke^{ikx}\in L^2(\mathbb{T})$, $\|\mathcal{T}_fu\|_{L^2}^2=\sum_{k=-\infty}^{+\infty} |u_k|^2=\|u\|_{L^2}^2$. As $\mathcal{T}_f$ is linear, this proves \textit{(i)}.

Property \textit{(ii)} comes from the fact that $\Pi(e^f_k)=e^f_k$ if $k\geq 0$, and $0$ otherwise.

Finally, if $v\in L^\infty$, then $v$ can be identified to a function $\tilde{v}$ such that $v=\tilde{v}$ almost everywhere and $|\tilde{v}(x)|\leq \|v\|_{L^\infty}$ for \emph{all} $x\in\mathbb{T}$. Now, if $x\in\mathbb{T}$, $\mathcal{T}_fv(e^{ix})=\mathcal{T}_f\tilde{v}(e^{ix})=\tilde{v}(e^{iy})$, where $y\in\mathbb{T}$ is such that $e^{iy}=e^{ix}f(e^{ix})$. This proves that $|\mathcal{T}_f\tilde{v}(e^{ix})|\leq \|v\|_{L^\infty}$, so the restiction of $\mathcal{T}_f$ to $L^\infty$ gives rise to a continuous endomorphism of $L^\infty$.

As we obviously have $\mathcal{T}_f1=1$, it remains to show that $\mathcal{T}_f(uv)=(\mathcal{T}_fu)(\mathcal{T}_fv)$ and $\mathcal{T}_f\bar{u}=\overline{\mathcal{T}_fu}$ whenever $u,v\in L^\infty$. But this is immediate once we have interpreted $\mathcal{T}_f$ as a composition operator.
\end{proof}

\begin{rem}
The converse of Lemma \ref{Tf_alg} is true : any operator $\mathcal{G}:L^2(\mathbb{T})\to L^2(\mathbb{T})$ that satisfies properties \textit{(i)}--\textit{(iii)} is of the form $\mathcal{T}_f$, where $f\in L^2_+$ is an inner function. It suffices to set $f(e^{ix}):=e^{-ix}\mathcal{G}(e^{ix})$ for $x\in \mathbb{T}$.
\end{rem}

\begin{proof}[Proof of Proposition \ref{invar}.]
Firstly, we note that since $u_0=\Pi(b)$ for some $b\in L^\infty(\mathbb{T})$, then $\mathcal{T}_fu_0=\mathcal{T}_f\Pi(b)=\Pi(\mathcal{T}_fb)\in BMO_+$. Indeed, $\mathcal{T}_fb\in L^\infty$ by the previous lemma.

If $u$ is the solution starting from $u_0$, then we get
\[\begin{aligned} i\partial_t(\mathcal{T}_fu)&=\mathcal{T}_f(i\partial_tu)\\
&=2J(u)\mathcal{T}_f\Pi(|u|^2)+\overline{J(u)}\mathcal{T}_f(u^2)\\
&=2J(u)\Pi(|\mathcal{T}_fu|^2)+\overline{J(u)}(\mathcal{T}_fu)^2.\end{aligned}\]
Furthermore, $J(u)=(u^2\vert u)=(\mathcal{T}_f(u^2)\vert \mathcal{T}_fu)=J(\mathcal{T}_fu)$, because of property \textit{(i)} of Lemma \ref{Tf_alg}. So $\mathcal{T}_fu$ solves equation \eqref{quad} in $BMO_+$, and $\mathcal{T}_f(u)(t=0)=\mathcal{T}_fu_0$.
\end{proof}

\begin{rem}
The mapping $\mathcal{T}_f$ also preserves the class of the standing waves in $BMO_+$, for if $\re u$ only takes two different values on $\mathbb{T}$, then so does $\re (\mathcal{T}_fu)$.
\end{rem}

\vspace*{3em}
\section{Stability and instability of the traveling waves}\label{stab}

Hereafter we consider the problem of the stability of the solutions that Theorem \ref{main} classifies.

\subsection{The ground state}
We begin with the ``ground state'', and show its stability via the functional inequality \eqref{gagliardo} as in \cite{BizonGround, ann}.

\begin{proof}[Proof of Proposition \ref{gagliardo-prop}.] Let $u\in H^{1/2}_+$, and write $u(z)=\sum_{k\geq 0}u_kz^{k}$. Then by Parserval's theorem and twice the Cauchy-Schwarz inequality,
\[\begin{aligned} E(u)=\frac{1}{2}|J(u)|^2&=\frac{1}{2}\left| \sum_{k,\ell\geq 0}u_ku_\ell\overline{u_{k+\ell}}\right|^2\\
&\leq \frac{1}{2}\left(\sum_{k\geq 0}|u_k|^2\right)\left(\sum_{k\geq 0}\left| \sum_{\ell \geq 0}u_\ell\overline{u_{k+\ell}}\right|^2\right) \\
&\leq \frac{1}{2}Q(u)\cdot \left(\sum_{\ell\geq 0}|u_\ell|^2\right)\left( \sum_{k,\ell\geq 0}|u_{k+\ell}|^2\right)\\
&=\frac{1}{2}Q(u)^2\left( \sum_{m\geq 0}(1+m)|u_m|^2\right) = \frac{1}{2}Q(u)^2(Q(u)+M(u)).
\end{aligned}\]
If in addition $2E=Q^2(Q+M)$, then the case of equality in Cauchy-Schwarz allows us to say that necessarily, $\forall k\geq 0$, there exists $\gamma_k\in\mathbb{C}$ such that
\[ u_\ell=\gamma_ku_{\ell+k},\quad \forall \ell\geq 0.\]
In particular, $u_\ell=\gamma_1u_{\ell+1}$, so $\{u_\ell\}$ must be a geometric sequence. There exists $\lambda,p\in\mathbb{C}$ such that
\[u_\ell=\lambda p^\ell,\quad\forall \ell\geq 0.\]
As $\sum_\ell |u_\ell|^2<\infty$, we further have $|p|<1$. Hence $u(z)=\frac{\lambda}{1-pz}$. Conversely, if $\{u_\ell\}$ is a geometric sequence, both inequalities in the previous computation become equalities.\end{proof}

Now it is standard to prove the $H^{1/2}$-stability of the traveling waves $z\mapsto\frac{\lambda}{1-pz}$ (which are even stationary waves in the case when $p=0$).

\begin{proof}[Proof of Corollary \ref{coro-stab}.] Let $\lambda,p\in\mathbb{C}$ with $|p|<1$, and $v_0(z)=\frac{\lambda}{1-pz}$. Simple computations show that
\[ Q(v_0)=\frac{|\lambda|^2}{1-|p|^2},\quad
M(v_0)=\frac{|\lambda|^2|p|^2}{(1-|p|^2)^2},\quad
E(v_0)=\frac{1}{2}\frac{|\lambda|^6}{(1-|p|^2)^4}.\]
Then by Proposition \ref{gagliardo-prop}, it appears that
\[ \{ e^{-i\theta}v_0(ze^{-i\alpha})\mid (\theta,\alpha)\in\mathbb{T}^2\}=\{u\in H^{1/2}_+\mid Q(u)=Q(v_0),\: M(u)=M(v_0),\: E(u)=E(v_0)\}.\]
In the sequel, we denote by $\mathcal{T}$ this two- (or maybe one-) dimensional torus, and we show that it is stable in the energy space.

Suppose that it is not. Then there exists a sequence $\{u_0^n\}$ of initial data in $H^{1/2}_+$ such that $d(u_0^n,\mathcal{T})\to 0$ (where the distance is taken in $H^{1/2}_+$), and a sequence of times $\{t^n\}$ such that the solution $u^n(t)$ of \eqref{quad} starting from $u_0^n$ at time $t=0$ satisfies
\begin{equation}\label{loin}
d(u^n(t^n),\mathcal{T})\geq \varepsilon_0>0,
\end{equation}
for some $\varepsilon_0>0$.

Because of the conservation laws, we know that $\{u^n(t^n)\}$ is bounded in $H^{1/2}_+$. Up to extracting a subsequence, there is an element $f$ of the energy space such that $u^n(t^n)\rightharpoonup f$. By Rellich's theorem, we know that $u^n(t^n)\to f$ strongly in any $L^p$, $p<\infty$, so
\begin{gather*}
Q(v_0)=\lim_{n\to \infty} Q(u_0^n)=\lim_{n\to\infty}Q(u^n(t^n))=Q(f),\\
E(v_0)=\lim_{n\to \infty} E(u_0^n)=\lim_{n\to\infty}E(u^n(t^n))=E(f).
\end{gather*}
By the weak convergence, we also have $M(f)\leq \liminf_{n\to\infty} M(u^n(t^n))=\liminf_{n\to\infty}M(u_0^n)=M(v_0)$. Hence
\[ E(f)\leq \frac{1}{2}Q(f)^2(Q(f)+M(f))\leq \frac{1}{2}Q(v_0)^2(Q(v_0)+M(v_0))=E(v_0)=E(f),\]
so all the inequalities are in fact equalities. Thus $M(f)=M(v_0)$, so $u^n(t^n)\to f$ strongly in $H^{1/2}$ and $f\in\mathcal{T}$. This is a contradiction with \eqref{loin}.
\end{proof}

\subsection{The second branch}
Now it remains to consider the case of the second branch of the traveling waves in Theorem \ref{main}. We first focus on the simplest one, of degree 1 in $z$, namely
\[ z\mapsto -\lambda\frac{1+|p|^2}{1-|p|^2}+\frac{\lambda}{1-pz}=-\lambda\frac{2|p|^2}{1-|p|^2}+\frac{\lambda p z}{1-pz},\]
where $\lambda,p\in\mathbb{C}\setminus\{0\}$, and $|p|<1$. Observe that thanks to time rescaling, and to the invariance of equation \eqref{quad} under phase rotation and phase translation, it is enough to study the stability of the following traveling waves :
\[ v_r: z\in\mathbb{D}\longmapsto -\frac{2r}{1-r}+\frac{z\sqrt{r}}{1-z\sqrt{r}},\quad r\in (0,1).\]

\begin{lemme}\label{instab-v3}
Let $0<r<1$ be fixed. Then $v_r$ is not $H^{1/2}_+$-orbitally stable.
\end{lemme}

Before starting the proof of this lemma, let us recall how equation \eqref{quad} looks like on $\mathcal{V}(3)$. The exact formulae can be found in Appendix \ref{formulaire}. Let $u$ be a solution of \eqref{quad} of the form
\[ u(z)=b+\frac{cz}{1-pz},\quad \forall z\in\mathbb{D},\]
with $b,c,p\in \mathbb{C}$, $c\neq 0$, $c-bp\neq 0$, $|p|<1$. Then $Q$, $M$ and $J$ can be computed explicitely in terms of $b$, $c$, $p$ :
\begin{gather*}
Q=|b|^2+\frac{|c|^2}{1-|p|^2}, \\
M=\frac{|c|^2}{(1-|p|^2)^2}, \\
J=|b|^2b+\frac{2b|c|^2}{1-|p|^2}+\frac{|c|^2c\bar{p}}{(1-|p|^2)^2}.
\end{gather*}
Since $|c|=\sqrt{M}\cdot(1-|p|^2)$ and $|b|^2=Q-|c|\sqrt{M}$, we can rewrite $J$ in a simpler way as
\[J=(Q+|c|\sqrt{M})b+Mc\bar{p}.\]
Introduce the doubled energy $\mathcal{E}:=2E=|J|^2$. From the expression of $J$, we find
\[\mathcal{E}=(Q+|c|\sqrt{M})^2(Q-|c|\sqrt{M})+M|c|^2(M-|c|\sqrt{M})+2M(Q+|c|\sqrt{M})\re (b\bar{c}p).\]
Writing for simplicity $x:=|c|\sqrt{M}$, and defining $\psi$ by $\psi=\arg (b\bar{c}p)$ whenever $bp\neq 0$, and $\psi=0$ when $b=0$ or $p=0$, we get
\begin{equation}\label{energy-v3}
\mathcal{E}=(Q+x)^2(Q-x)+x^2(M-x)+2x(Q+x)\sqrt{(Q-x)(M-x)}\cdot\cos \psi.
\end{equation}

As far as the evolution of $x$ is concerned, recall that
\[\begin{aligned}i\frac{dc}{dt} &=2bc\bar{J}+2\bar{b}cJ + \frac{2Jp|c|^2}{1-|p|^2}\\
&=\left[4\re (\bar{b}J)+2M\frac{|p|^2|c|^2}{1-|p|^2}\right]c+2|c|^2(Q+|c|\sqrt{M})\frac{bp}{1-|p|^2},\end{aligned}\]
hence
\[ \frac{d|c|}{dt}=\frac{1}{|c|}\re \left( \bar{c}\frac{dc}{dt}\right)=2|c|(Q+|c|\sqrt{M})\frac{\im (b\bar{c}p)}{1-|p|^2},\]
or equivalently,
\[\frac{dx}{dt}=2x(Q+x)\sqrt{(Q-x)(M-x)}\cdot\sin \psi.\]
Using \eqref{energy-v3}, and writing $(\sin \psi)^2=1-(\cos \psi)^2$, we can finally express the evolution of $x$ in terms of conservation laws only :
\begin{equation}\label{evol-x}
\left( \frac{dx}{dt}\right)^2=4x^2(Q+x)^2(Q-x)(M-x)-\left[(Q+x)^2(Q-x)+x^2(M-x)-\mathcal{E}\right]^2.
\end{equation}

\vspace{1em}
\begin{proof}[Proof of Lemma \ref{instab-v3}.]
Thanks to the above formulae, we can easily compute
\begin{align*} Q_r:=Q(v_r)=\frac{r}{(1-r)^2}(3r+1),\qquad M_r:=M(v_r)&=\frac{r}{(1-r)^2},\\
J_r:=J(v_r)=-\frac{r^2}{(1-r)^3}(5r+3),\qquad \mathcal{E}_r=|J_r|^2&=\frac{r^4}{(1-r)^6}(5r+3)^2.
\end{align*}
Besides, since $v_r$ gives rise to traveling wave solution, $|c|$ (hence $x$) will be constant, so $\forall t\in\mathbb{R}$, $x(t)=x_r:=\sqrt{r}\sqrt{M}=\frac{r}{1-r}$. For the same reason, $\psi (t)= \psi_r:=\pi$. Equation \eqref{evol-x} then gives
\begin{equation}\label{eq-xr}
0=4x_r^2(Q_r+x_r)^2(Q_r-x_r)(M_r-x_r)-\left[(Q_r+x_r)^2(Q_r-x_r)+x_r^2(M_r-x_r)-\mathcal{E}_r\right]^2.
\end{equation}

Now we are going to perturb equation \eqref{evol-x} around $x_r$ and $\mathcal{E}_r$ without changing $Q_r$ nor $M_r$. Let $\gamma>0$ be a small enough parameter (in a sense that we will further make clear), and set
\[ u_0^\gamma(z)=-e^{i\gamma}\frac{2r}{1-r}+\frac{z\sqrt{r}}{1-z\sqrt{r}}, \quad \forall z\in\mathbb{D}.\]
We have $\|u_0^\gamma-v_r\|_{H^{1/2}}=\frac{2r}{1-r}|1-e^{i\gamma}|=\mathcal{O}_r(\gamma)$, where the notation $\mathcal{O}_r$ means that the constant in the $\mathcal{O}$ only depends on $r$. In addition, $Q(u_0^\gamma)=Q_r$, $M(u_0^\gamma)=M_r$. Computing $\mathcal{E}(u_0^\gamma)$ thanks to \eqref{energy-v3}, we see that only the angle $\psi$ is modified at time $0$, and changes from $\pi$ to $\pi +\gamma$. If $\gamma$ is sufficiently small, we have $\cos (\pi +\gamma)>-1$, and thus we can write $\mathcal{E}(u_0^\gamma)=\mathcal{E}_r+\delta\mathcal{E}$, where $\delta\mathcal{E}>0$ and $\delta\mathcal{E}=\mathcal{O}_r(\gamma)$.

Denote by $u$ (with an implicit dependence on $\gamma$) the solution of \eqref{quad} starting from $u_0^\gamma$. Write $x=|c|\sqrt{M}$ as above, and decompose $x=x_r+y$, where $y$ is initially zero, and is meant to remain small. Equation \eqref{evol-x} yields
\begin{multline*}\left( \frac{dy}{dt}\right)^2=4(x_r+y)^2(Q_r+x_r+y)^2(Q_r-x_r-y)(M_r-x_r-y)\\
-\left[(Q_r+x_r+y)^2(Q_r-x_r-y)+(x_r+y)^2(M_r-x_r-y)-\mathcal{E}_r-\delta\mathcal{E}\right]^2.
\end{multline*}
We must develop this expression. The leading order term vanishes because of \eqref{eq-xr}. Handling the $\delta\mathcal{E}$-terms with care, we can write
\begin{align*}
\left(\frac{dy}{dt}\right)^2=\delta\mathcal{E}&\cdot\left[-2\mathcal{E}_r+2(Q_r+x_r)^2(Q_r-x_r)+2x_r^2(M_r-x_r)+\mathcal{O}_r(y)-\delta\mathcal{E}\right]\\
&+y \cdot \left[ \frac{d}{dy}\left. F(y)\right|_{y=0} \right] \\
&+\mathcal{O}_r(y^2),
\end{align*}
where
\begin{align*}
F(y)=4(x_r+&y)^2(Q_r+x_r+y)^2(Q_r-x_r-y)(M_r-x_r-y)\\
&+2\mathcal{E}_r\cdot \left[ (Q_r+x_r+y)^2(Q_r-x_r-y)+(x_r+y)^2(M_r-x_r-y)\right]\\
&-\left[ (Q_r+x_r+y)^2(Q_r-x_r-y)+(x_r+y)^2(M_r-x_r-y)\right]^2.
\end{align*}
Here, the important point is that all the $\mathcal{O}$ do not depend on $\gamma$ nor on $\delta\mathcal{E}$. After a tedious calculation, we get
\begin{equation}\label{evol-y}
\left(\frac{dy}{dt}\right)^2=\delta\mathcal{E}\cdot\left[\frac{16r^4(1+r)}{(1-r)^5}+\mathcal{O}_r(y)-\delta\mathcal{E}\right]+\left[- \frac{64r^7(1+r)^2}{(1-r)^9}+\mathcal{O}_r(y)\right] \cdot y.
\end{equation}

Now, assume that $v_r$ is $H^{1/2}$-orbitally stable. Take $\varepsilon>0$ to be adjusted, and the corresponding $\eta>0$. Choose $\gamma_*>0$ small so that $\forall \gamma\in (0,\gamma_*)$, we have $\|u_0^{\gamma}-v_r\|_{H^{1/2}}\leq \eta$, and consider $u$ the corresponding solution. With the above notations for $u$, observe that for any $(\theta,\alpha)\in\mathbb{T}^2$, and $t\in\mathbb{R}$, 
\begin{align*}
\|u(t,z)-e^{-i\theta}v_r(ze^{-i\alpha})\|_{H^{1/2}}&\geq \left|\left( u(t,z)-e^{-i\theta}v_r(ze^{-i\alpha})\middle| z\right) \right| \\
&=|c(t)-e^{-i(\theta+\alpha)}\sqrt{r}| \\
&\geq \left| |c(t)|-\sqrt{r}\right| =\frac{1}{\sqrt{M_r}}|y(t)|
\end{align*}
so $\forall t\in\mathbb{R}$,
\[|y(t)|\leq \sqrt{M_r}\cdot \inf_{(\theta,\alpha)\in\mathbb{T}^2} \|u(t)-e^{-i\theta}v_r(ze^{-i\alpha})\|_{H^{1/2}} \leq \sqrt{M_r} \cdot \varepsilon \]
because of the stability. Going back to \eqref{evol-y}, taking $\varepsilon$ small enough (in a way that only depends on $r$), and shrinking $\gamma_*$ if needed to lessen $\delta\mathcal{E}$, we get
\begin{gather*} \frac{16r^4(1+r)}{(1-r)^5}+\mathcal{O}_r(y)-\delta\mathcal{E} \geq a_r, \\
- \frac{64r^7(1+r)^2}{(1-r)^9}+\mathcal{O}_r(y)\leq -b_r,
\end{gather*}
for all times $t\in\mathbb{R}$, where $a_r,b_r>0$ are some constants only depending on $r$.

At time $t=0$, we have $y=0$, so
\[\left( \frac{dy}{dt}|_{t=0}\right)^2\geq a_r\delta\mathcal{E}>0.\]
We suppose that $dy/dt <0$ at time $t=0$ (it suffices to argue on negative times if the converse is true). Then $y<0$ in some maximal time interval $(0,T)$. Thus \eqref{evol-y} yields
\[ \left(\frac{dy}{dt}\right)^2\geq a_r\delta\mathcal{E}+b_r |y|\geq a_r\delta\mathcal{E} \quad \text{on }(0,T),\]
and since $y$ is $C^\infty$, we have
\[ \frac{dy}{dt}<-\sqrt{a_r\delta\mathcal{E}} \quad \text{on }(0,T).\]
This yields $T=+\infty$ and $\forall t\in\mathbb{R}_+$, $y(t)\leq -t\sqrt{a_r\delta\mathcal{E}}$, which cannot happen. This ends the proof of the instability of $v_r$.
\end{proof}

Let us now prove the instability of the traveling waves of higher degree.
\begin{proof}[Proof of Proposition \ref{instab}.] In Lemma \ref{instab-v3}, we have proved the existence of some $\varepsilon_0>0$, and constructed initial data $u_0^\gamma$ arbitrary close to $v_r$ in $H^{1/2}_+$, such that the corresponding solution $u$ satisfies $\|u(T,z)-e^{-i\theta}v_r(ze^{-i\alpha})\|_{H^{1/2}}\geq \varepsilon_0>0$ for some $T\in \mathbb{R}$, and all $(\theta,\alpha)\in \mathbb{T}^2$. If now $N\geq 2$, Proposition \ref{invar} precisely states that the solution of \eqref{quad} starting from $u_0^\gamma(z^N)$ is $u(t,z^N)$. On the one hand, we have
\[ \|u_0^\gamma(z^N)-v_r(z^N)\|_{H^{1/2}}=\|u_0^\gamma(z)-v_r(z)\|_{H^{1/2}}\ll 1,\]
since $u_0^\gamma$ only differs from $v_r$ by its constant term. On the other hand, for all $(\theta,\alpha)\in \mathbb{T}^2$,
\[ \|u(T,z^N)-e^{-i\theta}v_r\left((ze^{-i\alpha})^N\right)\|_{H^{1/2}}\geq \|u(T,z)-e^{-i\theta}v_r(ze^{-iN\alpha})\|_{H^{1/2}}\geq \varepsilon_0,\]
where we used the elementary fact that for any $w\in H^{1/2}_+$, 
\[ \|w(z^N)\|_{H^{1/2}}^2=\sum_{k\geq 0}(1+kN)|\hat{w}(k)|^2\geq \sum_{k\geq 0}(1+k)|\hat{w}(k)|^2 =\|w\|_{H^{1/2}}^2. \]
This concludes our proof.
\end{proof}

\appendix
\vspace*{2em}
\section*{Appendices}
\section{About equilibrium points}\label{steady}
The purpose of this section is to discuss the case of the steady solutions of \eqref{quad}. Notice first that if for some $u\in H^{1/2}_+$, $2J\Pi (|u|^2)+\bar{J}u^2=0$, then taking the scalar product with $u$, we find $3|J|^2=0$, hence $E=0$. Hence all steady solutions are issued from null-energy functions. Describing the set of steady solutions, say, in $H^{1/2}_+$ amounts to giving a characterization of the set $\{J=0\}$.

As an illustration of how tough this may be, we are now going to give an explicit description of the subset of the steady solutions that are also in $\mathcal{V}(3)$ thanks to the inverse transform of Gérard and Grellier \cite{livrePG}.

\begin{prop}
The subset of $\mathcal{V}(3)$ consisting in steady solutions of \eqref{quad}, called $\mathcal{V}_0(3)$, is a submanifold of real codimension $2$ given by
\begin{equation*}\mathcal{V}_0(3)= \left\lbrace \lambda e^{ia}\left( -\tfrac{2\sqrt{3}}{3}\sin \theta + \frac{\frac{(1+2\cos 2\theta )^2}{3\sqrt{9+2\cos 2\theta - 2\cos 4\theta}}ze^{ib}}{1-\frac{4 (2+\cos 2\theta )\sin \theta}{\sqrt{3} \sqrt{9+2\cos 2\theta - 2\cos 4\theta}}ze^{ib}} \right) \ \middle| \ \lambda\in \mathbb{R},\ a, b \in \mathbb{T},\theta \in [0,\tfrac{\pi}{3}) \right\rbrace  .
\end{equation*}
\end{prop}

\begin{proof}[Proof]
Let us fix $u$ a function of $\mathcal{V}(3)$ such that $J(u)=0$. We suppose that $u$ is not identically zero, and rescale it so that $Q(u)=1$. Denote by $\sigma^2$ the unique non-zero eigenvalue of $K_u^2$. We have $\sigma^2=\tr K_u^2=M$. Moreover, it is an elementary fact (see \cite{GGtori}) that eigenvalues of $H_u^2$ and $K_u^2$ are interlaced. As $\rg H_u^2 = 2$, we denote by $\rho_1^2$, $\rho_2^2$ its eigenvalues, in such a way that
\begin{equation}\label{interlaced}
\rho_1^2\geq \sigma^2 \geq \rho_2^2\geq 0.
\end{equation}
In addition, we have $\rho_1^2+\rho_2^2=\tr H_u^2=Q+M$, hence 
\begin{equation}\label{rho-sigma}
\rho_1^2-\sigma^2+\rho_2^2=Q=1.
\end{equation}

The case when $\rho_1=\rho_2$ can be treated first, because it corresponds to a Blaschke product of degree $1$. In that case, $u$ should be of the form
\[ u(z)=e^{ia}\frac{z-\bar{p}}{1-pz}, \quad\forall z\in\mathbb{D},\]
for some $p\in\mathbb{D}$ and $a\in\mathbb{T}$. However, for such a function, $J(u)=\int_\mathbb{T}|u|^2u=(u\vert 1)=-e^{ia}\bar{p}$. Therefore, $p=0$ and $u(z)=e^{ia}z$. From now on, we assume that $\rho_1>\rho_2$ in \eqref{interlaced}.

Let us now briefly recall some notations. We denote by $E_u(\rho_j)$ the eigenspace of $H_u^2$ associated to the eigenvalue $\rho_j^2$, $j\in \{1,2\}$, and we define $u_j$ to be the orthogonal projection of $u$ onto $E_u(\rho_j)$. Since $u=H_u(1)\in\ran H_u=\ran H_u^2$, we have $u=u_1+u_2$ with $u_1\perp u_2$. We can also characterize the action of $H_u$ on $u_j$, noting that $\dim E_u(\rho_j^2) = 1$ : there exists $\varphi_1,\varphi_2\in\mathbb{T}$ such that
\[ \left\lbrace \begin{aligned}
H_u(u_1)&=\rho_1e^{i\varphi_1}u_1, \\
H_u(u_2)&=\rho_2e^{i\varphi_2}u_2.
\end{aligned}\right. \]
Last of all, it is possible to compute the $L^2$ norm of the $u_j$'s in terms of $\rho_1^2$, $\sigma^2$, $\rho_2^2$ (see \emph{e.g. }\cite{Xu3}) :
\[ \|u_1\|^2=\rho_1^2\frac{\rho_1^2-\sigma^2}{\rho_1^2-\rho_2^2}, \quad \|u_2\|^2=\rho_2^2\frac{\sigma^2-\rho_2^2}{\rho_1^2-\rho_2^2}.
\]

Using the decomposition $u=u_1+u_2$, we get
\[ 0=J(u)=(u\vert H_u(u))=\rho_1e^{-i\varphi_1}\|u_1\|^2+\rho_2e^{-i\varphi_2}\|u_2\|^2.\]
Up to a rotation of $u$, we can assume that $\varphi_1=0$. As $\|u_2\|^2\neq 0$, $e^{i\varphi_2}$ must then be real, so necessarily, $\varphi_2=\pi$. Now using the expression of $\|u_1\|^2$, $\|u_2\|^2$, and multiplying by $\rho_1^2-\rho_2^2$, we get
\[ \rho_1^3(\rho_1^2-\sigma^2)-\rho_2^3(\sigma^2-\rho_2^2)=0.\]
By \eqref{rho-sigma}, $\rho_1^2-\sigma^2=1-\rho_2^2$ and $\sigma^2-\rho_2^2=\rho_1^2-1$. Hence $\rho_1^3+\rho_2^3=\rho_1^2\rho_2^2(\rho_1+\rho_2)$, or rather
\[ \rho_1^2-\rho_1\rho_2+\rho_2^2=\rho_1^2\rho_2^2. \]
This equation means that $x:=\frac{1}{\rho_1}$ and $y:=\frac{1}{\rho_2}$ belong to the ellipse of equation $x^2-xy+y^2=1$. A standard reduction procedure indicates that all the points of this ellipse can be written as $(x,y)=\frac{2}{\sqrt{3}}(\cos (\theta +\frac{\pi}{6}),\cos (\theta -\frac{\pi}{6}))$, for some $\theta\in\mathbb{R}$. However, we must have $y> x$ and $x,y> 0$, which imposes $\theta \in (0,\frac{\pi}{3})$.

The matrix involved in the inverse spectral formula reads
\[ \mathscr{C}(z):=\begin{pmatrix}
\dfrac{\rho_1-\sigma e^{i\psi}z}{\rho_1^2-\sigma^2} & \dfrac{1}{\rho_1}\\
\dfrac{\rho_2+\sigma e^{i\psi}z}{\sigma^2 -\rho_2^2} & -\dfrac{1}{\rho_2}
\end{pmatrix}, \quad z\in\mathbb{D}.
\]
A translation in $z$ enables to reduce ourselves to the case when the angle associated to $\sigma^2/2$, called $\psi$, is zero. Then we can reconstruct
\[ u(z)=\left\langle \mathscr{C}(z)^{-1}\begin{pmatrix}
1 \\ 1
\end{pmatrix}, \begin{pmatrix}
1 \\ 1
\end{pmatrix} \right\rangle_{\mathbb{C}^2}=\tfrac{1-\rho_1\rho_2}{\rho_1-\rho_2} - \frac{z}{\frac{\sigma(\rho_1-\rho_2)^2}{1-\rho_1\rho_2}+(1+\rho_1\rho_2)(\rho_1-\rho_2)z},\]
and $\sigma=\sqrt{\rho_1^2+\rho_2^2-1}$, which leads to the claim, by means of a few trigonometric identities.
\end{proof}

\begin{eg}
The function
\[ u(z)= -\frac{\sqrt{3}}{3}+\frac{\frac{4}{3\sqrt{11}}z}{1-\frac{5}{\sqrt{33}}z}\]
corresponds to an equilibrium point of \eqref{quad}.
\end{eg}

\section{Formulae for solutions in \texorpdfstring{$\mathcal{V}(3)$}{V3}}\label{formulaire}

We summarize below some useful formulae from \cite[Section 4]{thirouin2} for solutions of \eqref{quad} of the form
\[ u(z)=b+\frac{cz}{1-pz},\quad\forall z\in\mathbb{D},\]
where $b,c,p\in\mathbb{C}$, and $|p|<1$.

We have
\begin{gather*}
Q=|b|^2+\frac{|c|^2}{1-|p|^2}, \\
M=\frac{|c|^2}{(1-|p|^2)^2}, \\
J=|b|^2b+\frac{2b|c|^2}{1-|p|^2}+\frac{|c|^2c\bar{p}}{(1-|p|^2)^2}.
\end{gather*}

As far as the evolution of $u$ is concerned, \eqref{quad} translates into
\[\left\lbrace
\begin{aligned}
i\dot{p} &=c\bar{J},\\
i\dot{c} &=2bc\bar{J}+2\bar{b}cJ + \frac{2Jp|c|^2}{1-|p|^2},\\
i\dot{b} &=b^2\bar{J}+2|b|^2J+\frac{2J|c|^2}{1-|p|^2}.
\end{aligned}\right.
\]

\vspace{0,5cm}
\bibliography{mabiblio}

\begin{thebibliography}{10}

\bibitem{a-toland-2}
C.~J. Amick and J.~F. Toland.
\newblock Uniqueness and related analytic properties for the {B}enjamin-{O}no
  equation---a nonlinear {N}eumann problem in the plane.
\newblock {\em Acta Math.}, 167(1-2):107--126, 1991.

\bibitem{a-toland}
C.~J. Amick and J.~F. Toland.
\newblock Uniqueness of {B}enjamin's solitary-wave solution of the
  {B}enjamin-{O}no equation.
\newblock {\em IMA J. Appl. Math.}, 46(1-2):21--28, 1991.
\newblock The Brooke Benjamin special issue (University Park, PA, 1989).

\bibitem{Bizon}
P.~Bizo\'n, B.~Craps, O.~Evnin, D.~Hunik, V.~Luyten, and M.~Maliborski.
\newblock Conformal flow on {$\mathbb S^3$} and weak field integrability in
  {$\mathrm{AdS}_4$}.
\newblock {\em Comm. Math. Phys.}, 353(3):1179--1199, 2017.

\bibitem{BizonGround}
P.~Bizo{\'n}, D.~Hunik-Kostyra, and D.~Pelinovsky.
\newblock Ground state of the conformal flow on {$\mathbb S^3$}.
\newblock {\em arXiv preprint arXiv:1706.07726}, 2017.

\bibitem{LLL}
P.~G{\'e}rard, P.~Germain, and L.~Thomann.
\newblock On the cubic lowest {L}andau level equation.
\newblock {\em arXiv preprint arXiv:1709.04276}, 2017.

\bibitem{ann}
P.~G{\'e}rard and S.~Grellier.
\newblock The cubic {S}zeg{\H o} equation.
\newblock {\em Ann. Sci. \'Ec. Norm. Sup\'er. (4)}, 43(5):761--810, 2010.

\bibitem{GGtori}
P.~G{\'e}rard and S.~Grellier.
\newblock Invariant tori for the cubic {S}zeg{\H o} equation.
\newblock {\em Invent. Math.}, 187(3):707--754, 2012.

\bibitem{explicit}
P.~G{\'e}rard and S.~Grellier.
\newblock An explicit formula for the cubic {S}zeg{\H{o}} equation.
\newblock {\em Transactions of the American Mathematical Society},
  367(4):2979--2995, 2015.

\bibitem{livrePG}
P.~G\'erard and S.~Grellier.
\newblock The cubic {S}zeg{\H o} equation and {H}ankel operators.
\newblock {\em Ast\'erisque}, (389):vi+112, 2017.

\bibitem{GKoch}
P.~G\'erard and H.~Koch.
\newblock The cubic {S}zeg{\H o} flow at low regularity.
\newblock In {\em S\'eminaire {L}aurent {S}chwartz---\'Equations aux
  d\'eriv\'ees partielles et applications. {A}nn\'ee 2016--2017}, pages Exp No.
  XIV, 14 p. Ed. \'Ec. Polytech., Palaiseau, 2017.

\bibitem{PocoSol}
O.~Pocovnicu.
\newblock Traveling waves for the cubic {S}zeg{\H o} equation on the real line.
\newblock {\em Anal. PDE}, 4(3):379--404, 2011.

\bibitem{PocoToeplitz}
O.~Pocovnicu.
\newblock Soliton interaction with small {T}oeplitz potentials for the
  {S}zeg\"o equation on {$\mathbb R$}.
\newblock {\em Dyn. Partial Differ. Equ.}, 9(1):1--27, 2012.

\bibitem{Rudin}
W.~Rudin.
\newblock {\em Real and complex analysis}.
\newblock McGraw-Hill Book Co., New York, third edition, 1987.

\bibitem{thirouin2}
J.~Thirouin.
\newblock Optimal bounds for the growth of {S}obolev norms of solutions of a
  quadratic {S}zeg{\H o} equation.
\newblock {\em to appear in Transactions of the Am. Math. Soc.}, 2017.

\bibitem{Wu-BO}
Y.~Wu.
\newblock Simplicity and finiteness of discrete spectrum of the
  {B}enjamin-{O}no scattering operator.
\newblock {\em SIAM J. Math. Anal.}, 48(2):1348--1367, 2016.

\bibitem{Xu3}
H.~Xu.
\newblock The cubic {S}zeg{\H {o}} equation with a linear perturbation.
\newblock {\em arXiv preprint arXiv:1508.01500}, 2015.

\end{thebibliography}
\bibliographystyle{plain}

\vspace{0,5cm}
\textsc{Département de mathématiques et applications, École normale supérieure,
CNRS, PSL Research University, 75005 Paris, France}

\textit{E-mail address: }\texttt{joseph.thirouin@ens.fr}

\end{document}